\documentclass[a4paper,11pt]{article}

\usepackage{color}
\usepackage[latin1]{inputenc}
\usepackage[T1]{fontenc}
\usepackage[normalem]{ulem}
\usepackage[english]{babel}
\usepackage{verbatim}
\usepackage{graphicx}
\usepackage{enumerate}
\usepackage{amsmath,amssymb,amsfonts,amsthm}
\usepackage{rotating}
\usepackage{float}
\usepackage{array}
\usepackage{relsize, enumitem}

\newtheorem{theorem}{Theorem}[subsection]

\newtheorem*{theorem*}{Theorem}
\newtheorem*{cor*}{Corollary}
\newtheorem*{theoremBI}{Theorem $B^I$}

\newtheorem*{theoremBIG}{Theorem $B_{G}^I$}
\newtheorem{lemma}[theorem]{Lemma}
\newtheorem{prop}[theorem]{Proposition}
\newtheorem{cor}[theorem]{Corollary}

\theoremstyle{definition}
\newtheorem{defn}[theorem]{Definition}
\newtheorem{rem}[theorem]{Remark}
\newtheorem{ex}[theorem]{Example}

\input xy
\xyoption{all}

\DeclareMathOperator{\id}{id}

\DeclareMathOperator{\map}{Map}

\DeclareMathOperator{\Top}{Top}
\DeclareMathOperator{\Ob}{Ob}

\DeclareMathOperator*{\colim}{colim}
\DeclareMathOperator*{\hocolim}{hocolim}
\DeclareMathOperator*{\holim}{holim}

\DeclareMathOperator{\THH}{THH}

\DeclareMathOperator{\hofib}{hof}

\DeclareMathOperator{\conn}{Conn}

\DeclareMathOperator{\res}{res}

\DeclareMathOperator{\Pa}{Part}
\DeclareMathOperator{\hof}{hof}
\DeclareMathOperator{\fib}{fib}
\DeclareMathOperator{\Eff}{Eff}
\DeclareMathOperator{\Hom}{Hom}
\DeclareMathOperator{\conf}{Conf}

\newcommand{\bigslant}[2]{{\left.\raisebox{.2em}{$#1$}/\raisebox{-.2em}{$#2$}\right.}}
\newcommand{\unders}{\!<\!}

\begin{document}
\begin{center}\LARGE{Equivariant diagrams of spaces}
\end{center}

\begin{center}\large{Emanuele Dotto}

\end{center}
\vspace{.3cm}

\begin{quote}
\textsc{Abstract}. We generalize two classical homotopy theory results, the Blakers-Massey Theorem and Quillen's Theorem $B$, to $G$-equivariant cubical diagrams of spaces, for a discrete group $G$. We show that the equivariant Freudenthal suspension Theorem for permutation representations is a direct consequence of the equivariant Blakers-Massey Theorem. We also apply this theorem to generalize to $G$-manifolds a result about cubes of configuration spaces from embedding calculus. Our proof of the equivariant Theorem $B$ involves a generalization of the classical Theorem $B$ to higher dimensional cubes, as well as a categorical model for finite homotopy limits of classifying spaces of categories.
\end{quote}

\section*{Introduction}

Equivariant diagrams of spaces and their homotopy colimits have broad applications throughout topology. They are used in \cite{Jackowsky} for decomposing classifying spaces of finite groups, to study posets of $p$-groups in \cite{TW}, for splitting Thom spectra in \cite{Takayasu}, and even in the definition of the cyclic structure on $\THH$ of \cite{BHM}. 
In previous joint work with K. Moi \cite{Gdiags} the authors develop an extensive theory of equivariant diagrams in a general model category, and they study the fundamental properties of their homotopy limits and colimits.
In the present paper we restrict our attention to $G$-diagrams in the category of spaces. The special feature of $G$-diagrams of spaces is the existence of generalized fixed point functors that preserve and reflect equivalences. We study these functors and we use them to generalize the Blakers-Massey Theorem and Quillen's Theorem $B$ to equivariant cubical diagrams. 

Throughout the paper $G$ is going to be a discrete group. Let $I$ be a small category with a $G$-action. A $G$-structure on a diagram $X\colon I\to \Top$ is a sort of generalized $G$-action on $X$, which depends on the way $G$ acts on $I$ (see \ref{defGdiag}). The key feature of a $G$-structure is that it induces a $G$-action on the homotopy limit and on the homotopy colimit of $X$ (\S\ref{secGdiags}). These equivariant constructions can be described as derived functors in a suitable model categorical context (\cite{Gdiags}). In the present paper, we will focus mostly on $G$-diagrams of cubical shape.
If $J$ is a finite $G$-set, the poset category of subsets $I=\mathcal{P}(J)$ ordered by inclusion inherits a $G$-action. A $J$-cube is a diagram $X\colon \mathcal{P}(J)\to C$ equipped with a $G$-structure. There canonical maps
\[\phi\colon X_{\emptyset}\longrightarrow \holim_{\mathcal{P}_0(J)}X
\ \ \ \ \ \ \ \ \ \ \ \ \ \ \ \
\psi\colon \hocolim_{\mathcal{P}_1(J)}X\longrightarrow X_{J}\]
are $G$-equivariant, where $\mathcal{P}_0(J)$ and $\mathcal{P}_1(J)$ are the category $\mathcal{P}(J)$ respectively with the initial and the final object removed. Given a function $\nu\colon\{H\leq G\}\to\mathbb{Z}$ which is invariant on conjugacy classes, we say that $X$ is $\nu$-cartesian if the restriction of $\phi$ on $H$-fixed points is $\nu(H)$-connected. Dually, $X$ is $\nu$-cocartesian if $\psi$ is $\nu(H)$-connected on $H$-fixed points. The following generalizes the Blakers-Massey Theorem of \cite[2.5]{calcII}. We prove it in \S\ref{secBM}.

\begin{theorem*}[Equivariant Blakers-Massey]
Let $X\colon\mathcal{P}(J)\to\Top$ be a $J$-cube of spaces, and suppose that for every subgroup $H$ of $G$ and every non-empty $H$-invariant subset $U$ of $J$ the restriction $X|_{\mathcal{P}(U)}$ is $\nu^U$-cocartesian. Suppose moreover that these functions satisfy $\nu^U\leq \nu^V$ whenever $U\subset V$ are non-empty $H$-subsets of $J$. Then $X$ is $\nu$-cartesian, where $\nu$ is the function
\[\nu(H)=\!\min\left\{\!\min_{\{T_\alpha\}\in\Pa_H(J)}\{
\sum_{\alpha}\nu^{T_\alpha}(H)\}-|J/H|+1\ ,\ \min_{\emptyset\neq U\subset J}\min_{L\in \Eff_H(U)}\big\{\conn X^{L}_U-|U/L|+1\big\}
\right\}\]
\end{theorem*}
Here $\Pa_H(J)$ is the set of partitions of $J$ by $H$-invariant subsets, and $\Eff_H(U)$ is roughly the set of proper subgroups $L$ of $H$ for which $U/L\neq U/H$. The first term of the minimum corresponds to the standard range of the Blakers-Massey Theorem \cite[2.5]{calcII}. The second term is purely equivariant, and it is infinite if $J$ has trivial $G$-action. It comes form the equivariant connectivity of a certain space of natural transformations of diagrams, which is calculated in \ref{connholim} using a diagramatic obstruction theory argument.
We also prove a dual form of this Blakers-Massey Theorem in \ref{dualBM}.
In the same way that the Freudenthal suspension Theorem is an immediate consequence of the Blakers-Massey Theorem for the square
\[\xymatrix@=13pt{X\ar[r]\ar[d]&CX\ar[d]\\
CX\ar[r]&\Sigma X
}\]
the equivariant Freudenthal suspension Theorem follows from the equivariant Blakers-Massey Theorem applied to a certain equivariant cube, whose all but the initial and the final vertex are contractible. Given a finite $G$-set $J$ and a pointed $G$-space $X$, let $\Sigma^{J}X$ and $\Omega^JX$ be respectively the suspension and the loop space of $X$ by the permutation representation of $J$. We prove the following Corollary in \S\ref{secsusp}.

\begin{cor*}[Equivariant Suspension Theorem, \cite{Nam}] Let $X$ be a pointed $G$-space. The unit of the $(\Sigma^{J},\Omega^{J})$-adjunction restricted on $G$-fixed points $\eta\colon X^G\to (\Omega^J\Sigma^{J}X)^G$ is
\[\min\Big\{2\conn X^G+1,\min_{\substack{H\leq G\\ J/H\neq J/G}}\conn X^H\Big\}\]
connected.
\end{cor*}

As a second application of the equivariant Blakers-Massey Theorem, we prove an equivariant version of the relative disjunction Theorem of \cite{GK} for equivariant configuration spaces. Let $M$ be a manifold with a proper $G$-action. The space of configurations of $J$-points in $M$ is the space $\conf(J,M)$ of injective maps $J\rightarrowtail M$. This space inherits a $G$-action by conjugation, and it is functorial in the $J$-variable with respect to injective $G$-maps. The following is proved in \S\ref{secconf}.

\begin{cor*}
Let $J$ be a finite $G$-set and let $J_+$ be the $G$-set $J$ with an added fixed basepoint. The diagram $\conf(J_+\backslash (-),M)\colon \mathcal{P}(J_+)\to Top$ has a canonical $G$-structure, and it is $\nu$-cartesian for the function
\[\nu(H)=\min\Big\{|J|m_H-2|J/H|+1\ ,
\min_{\substack{
L\leq H\\
J/L\neq J/H
}}\big\{\min\{\conn M^L+1,m_L\}-|J/L|\big\}
 \Big\}\]
where $m_H$ is the dimension of the fixed points manifold $M^H$.
\end{cor*}
We prove in fact a stronger statement involving suitably transverse families of submanifolds of $M$ (Theorem \ref{intsubman}), and we deduce this Corollary from the case where all of these submanifolds are points.
In \ref{exsharp} we find an example where this range is sharp and it is determined by the second term of the minimum.

We turn to another classical result in homotopy theory, the celebrated Theorem $B$ of Quillen, from \cite{Quillen}. This theorem shows that under certain conditions the homotopy fiber of the geometric realization of a functor is itself the geometric realization of a category. It is not immediately clear how to generalize this result equivariantly. The analogous statement for an equivariant functor between categories with $G$-actions can easily be reduced to Theorem $B$ by taking fixed points. In order to achieve an interesting equivariant Theorem $B$, we need to extend it first to higher dimensional cubes. Let $J$ be a finite $G$-set and let $X\colon \mathcal{P}(J)\to Cat$ be a cube of categories. The natural transformations of diagrams of categories from $\mathcal{P}_0(-)\colon\mathcal{P}_0(J)\to Cat$ to $X$ form a category $\Hom\big(\mathcal{P}_0(-),X\big)$. A $G$-structure on $X$ induces a $G$-action on the category $\Hom\big(\mathcal{P}_0(-),X\big)$ by conjugation. The inclusion of objects as natural transformations of constant functors defines a functor $m_\emptyset\colon X_\emptyset\to \Hom\big(\mathcal{P}_0(-),X\big)$. Given a natural transformation $\Phi\colon \mathcal{P}_0(-)\to X$, the over category $m_\emptyset/_{\Phi}$ has an action of the stabilizer group $G_{\Phi}$. We proved the following result in \ref{GQuillen}.

\begin{theorem*}
Let $X\colon \mathcal{P}(J)\to Cat$ be a $J$-cube of categories, which satisfies a certain ``weak Reedy fibrancy condition'' (Definition \ref{defGreedy}). For every natural transformation $\Phi\colon \mathcal{P}_0(-)\to X$ the classifying space of the category $m_{\emptyset}/_{\Phi}$ is $G_{\Phi}$-equivalent to the total homotopy fiber of the cube of spaces $BX$ over $B\Phi$. In particular if all the categories $m_{\emptyset}/_{\Phi}$ are $G_{\Phi}$-contractible, $BX\colon \mathcal{P}(J)\to \Top$ is a homotopy cartesian $J$-cube of spaces.
\end{theorem*}

When $G$ is the trivial group and $J=1$ is the set with one element this is precisely Quillen's Theorem $B$. If $J=2$ is the set with two elements this is essentially Barwick and Kan's Quillen Theorem $B_2$ for homotopy pullbacks from \cite{ClarkKan}. To the best of the authors knowledge this is a new result for larger $J$, even for the trivial group. The key for proving this theorem is to define a good model for the homotopy limit of a diagram of categories. Let $I$ be a small category with finite dimensional nerve. Taking the over categories of $I$ defines a diagram of categories $I/_{(-)}\colon I\to Cat$. Given a functor $X\colon I\to Cat$, the natural transformations $\Hom\big(I/_{(-)},X\big)$ form a category, whose nerve is the Bousfield-Kan formula for the homotopy limit of the nerve of $X$. The following is proved in \S\ref{secBI}, and its equivariant version in \S\ref{secBGI}.

\begin{theoremBI}
Let $I$ be a category with finite dimensional nerve, and let $X\colon I\to Cat$ be a ``Reedy quasi-fibrant diagram'' (Definition \ref{quasiReedy}). The classifying space $B\Hom\big(I/_{(-)},X\big)$ is the homotopy limit of the diagram of spaces $BX\colon I\to\Top$.
\end{theoremBI}

\subsection*{Acknowledgments}
This project is a continuation of the work of \cite{Gdiags}, and part of a mission aimed at understanding a ``genuine'' context for equivariant homotopy theory. I wish to thank Kristian Moi for the many inspiring conversations we had over the last couple of years.
I also want to thank Brian Munson and Ismar Volic for useful correspondences.

\tableofcontents

\section{Preliminaries on equivariant diagrams}

\subsection{Equivariant diagrams of spaces and their fixed points}\label{secGdiags}

Let $G$ be a discrete group. A category with $G$-action is a functor $a\colon G\to Cat$, where the group $G$ is seen as a category with a unique object $\ast$. By abuse of notation we will refer to the underlying category $I=a(\ast)$ as a category with $G$-action.

\begin{defn}[\cite{vilf}]\label{defGdiag} Let $I$ be a small category with $G$-action, and let $C$ be a possibly large category. A $G$-structure on a diagram $X\colon I\to C$ is a collection of natural transformations $\phi_g\colon X\to X\circ g$, for every $g$ in $G$, subject to the following axioms.
\begin{enumerate}[label=\roman*)]
\item Let $1$ be the unit of $G$. Then $\phi_1$ is the identity natural transformation on $X$,
\item For every $h$ and $g$ in $G$ the diagram \ \ \  $\vcenter{\hbox{\xymatrix@=20pt{X\ar[dr]_-{\phi_{hg}}\ar[r]^-{\phi_g}&X\circ g\ar[d]^-{\phi_h|_g}\\
&X\circ h\circ g}}}$
\ \ \ commutes. 
\end{enumerate}
A diagram $X\colon I\to C$ equipped with a $G$-structure is called a $G$-diagram.
A morphism of $G$-diagrams is a natural transformation of underlying diagrams $f\colon X\to Y$ such that the square
\[\xymatrix@R=15pt{
X\ar[r]^-f\ar[d]_{\phi_g}&Y\ar[d]^-{\phi_g}\\
X\circ g \ar[r]_-{f|_g}& Y\circ g
}\]
commutes for every $g$ in $G$. Here $f|_g$ is the restriction of $f$ along the functor $g\colon I\to I$. The resulting category of $G$-diagrams is denoted $C_{a}^I$. We will often abuse the notation and write $g$ for the natural transformation $\phi_g$.
\end{defn}

\begin{ex}\label{overcatGdiag} Let $I$ and $J$ be two categories with $G$-action, and let $F\colon I\to J$ be an equivariant functor. The functor $F/(-)\colon J\to Cat$ which sends an object $j$ to the over category $F/j$ has a natural $G$-structure. The natural transformation $\phi_g$ is defined by the functors $F/j\to F/gj$ which send an object $(i\in I, \alpha\colon F(i)\to j)$ to $(gi,g\alpha\colon F(gi)=gF(i)\to gj)$.
By applying the classifying space functor we obtain a $G$-diagram  $B(F/-)\colon J\to \Top$ in the category of compactly generated Hausdorff spaces $\Top$.
\end{ex}

The category of $G$-diagrams of spaces $\Top_{a}^I$ is enriched in the category of $G$-spaces $\Top^G$. Given two $G$-diagrams $X$ and $Y$, the space of all natural transformations of underlying diagrams $\Hom_I(X,Y)$ inherits a $G$-action by conjugation:
\[g\cdot f=\big(X\stackrel{\phi_{g^{-1}}}{\xrightarrow{\hspace*{1cm}}}X\circ g^{-1}\stackrel{f|_{g^{-1}}}{\xrightarrow{\hspace*{1cm}}}Y\circ g^{-1}\stackrel{\phi_{g}}{\xrightarrow{\hspace*{1cm}}}Y\big)\]
The fixed points space $\Hom_I(X,Y)^G$ has the set of morphisms of $G$-diagrams $\Top_{a}^I(X,Y)$ as underlying set. 

The $G$-space $\Hom_I(X,Y)$ can be defined as a certain equalizer of $G$-spaces, and its construction can be dualized as follows. The $G$-action on $I$ induces a $G$-action on $I^{op}$. Given a $G$-diagram $X\colon I\to\Top$ and a $G$-diagram $Y\colon I^{op}\to\Top$, we define the coequalizer of $G$-spaces
\[X\otimes_I Y=\colim\left(\coprod_{\substack{\alpha\colon i\rightarrow j\\
\in \hom I}}Y_i\times X_j\rightrightarrows \coprod_{i\in \Ob I}Y_i\times X_i\right)\]
The maps are the standard maps of the Bousfield-Kan formula, see e.g. \cite[18.3.2]{hirsch}. 
The $G$-action on the target of the maps sends $(y,x)$ in $Y_i\times X_i$ to $(\phi_g(y),\phi_g(x))$ in $Y_{gi}\times X_{gi}$. The action on the source space is defined by a similar indexed coproduct.

\begin{defn}
Let $I$ be a category with $G$-action and $X\in \Top_{a}^I$ a $G$-diagram of spaces. The $G$-homotopy limit and the $G$-homotopy colimit of $X$ are the $G$-spaces defined respectively by
\[\holim_I X=\Hom_I\big(B(I/-),X\big)\ \ \ \ \ \ \ \ \ \hocolim_I X=X\otimes_I\big( B(-/I)^{op}\big) \]
\end{defn}

The underlying space of the $G$-homotopy limit is the homotopy limit of the underlying diagram, defined via the Bousfield-Kan formula, and dually for the $G$-homotopy colimit.

We recall that an equivariant map of $G$-spaces $f\colon X\to Y$ is an equivalence (in the fixed points model structure) if its restriction to the $H$-fixed points $f\colon X^H\to Y^H$ is a weak homotopy equivalence of spaces for every subgroup $H$ of $G$.
The vertex $X_i$ of a $G$-diagram $X$ inherits an action of the stabilizer group $G_i$ of the object $i\in I$, defined by the maps $\phi_g\colon X_i\to X_{gi}=X_i$.

\begin{defn}[\cite{vilf}]\label{defeqGdiag}
A morphism of $G$-diagrams $f\colon X\to Y$ is a weak equivalence if for every object $i$ of $I$ the map $f_i\colon X_i\to Y_i$ is an equivalence of $G_i$-spaces.
\end{defn}

The homotopical properties of the $G$-homotopy limit and of the $G$-homotopy colimit were studied extensively in \cite{Gdiags}. In particular both constructions send equivalences of $G$-diagrams to equivalences of $G$-spaces.

\begin{rem}\label{holimGrot} Let $G\ltimes_a I$ be the Grothendieck construction of the functor $a\colon G\to Cat$ which defines the $G$-action on $I$.
A $G$-structure on a diagram $X\colon I\to \Top$ is equivalent to an extension of $X$ to $G\ltimes_a I$ along the projection map $G\ltimes_a I\to I$. This results into an isomorphism of categories $\Top_{a}^I\cong \Top^{G\ltimes_a I}$ (see \cite[1.9]{Gdiags}). There is a relationship between the $G$-homotopy limit of a $G$-diagram $X\in \Top_{a}^I$, and the homotopy limit of the corresponding diagram $\overline{X}\colon G\ltimes_a I\to Top$. The latter computes the homotopy fixed points of the former:
\[\holim_{G\ltimes_a I}\overline{X}\simeq\big(\holim_I X\big)^{hG}\]
Dually, the homotopy colimit of $\overline{X}$ is described by the homotopy orbits
\[\hocolim_{G\ltimes_a I}\overline{X}\simeq\big(\hocolim_I X\big)_{hG}\]
These equivalences are an immediate consequence of the Fubini Theorems \cite[26.5]{CS}.
Homotopy orbits and homotopy fixed points are homotopy invariant with respect to na\"{i}ve equivalences of $G$-spaces ($G$-maps whose underlying map is an equivalence of spaces). We see from the formula above that $\holim_{G\ltimes_a I}\overline{X}$ and $\hocolim_{G\ltimes_a I}\overline{X}$ are invariant for the pointwise na\"{i}ve equivalences of $G$-diagrams.
Thus the $G$-homotopy limit and the $G$-homotopy colimit retain more equivariant information than the homotopy limit and homotopy colimit of $\overline{X}$. The categorical (in opposition to homotopy) fixed points of the $G$-homotopy limits and colimits are the focus of the next two sections.
\end{rem}

There is a notion of fixed points of a $G$-diagram. Let $H$ be a subgroup of $G$, and let $I^H$ be the subcategory of $I$ of objects and morphisms that are fixed (strictly) by the $H$-action. Equivalently, $I^H$ is the limit of the functor $H\to G\stackrel{a}{\to}Cat$. Since the $i$-vertex of a $G$-diagram $X\in \Top^{I}_a$ has an action of $G_i$, if $i$ belongs to $I^H$ the space $X_i$ has an $H$-action.

\begin{defn}\label{deffixeddiag}
Let $H$ be a subgroup of $G$.
The $H$-fixed points diagram of a $G$-diagram $X\in \Top^{I}_a$ is the diagram of spaces $X^H\colon I^H\to \Top$ of pointwise fixed points $(X^H)_i=(X_i)^H$. Given a map $\alpha\colon i\to j$ in $I^H$, the corresponding map $X_{i}^H\to X_{j}^H$ is the restriction of $\alpha_\ast\colon X_i\to X_j$ on $H$-fixed points.
\end{defn}

The diagram $X^H\colon I^H\to \Top$ is well defined on morphism because for every $h$ in $H$ the diagram
\[\xymatrix@R=15pt{X_{i}\ar[r]^-{\alpha_\ast}\ar[d]_-{\phi_h}&X_j\ar[d]^-{\phi_h}\\
X_{hi}\ar[r]_-{(h\alpha)_\ast}&X_{hj}
}\]
commutes by naturality of $\phi_h$. Hence if $\alpha$ is a morphism of $I^H$, the map $\alpha_\ast\colon X_i\to X_j$ is $H$-equivariant and it can be restricted on $H$-fixed points.

\begin{rem}
If $f\colon X\to Y$ is a morphism of $G$-diagrams, the map $f_i\colon X_i\to Y_i$ is $G_i$-equivariant for every object $i$ of $I$. Therefore $f$ restricts to a natural transformation $f^H\colon X^{H}\to Y^H$ for every subgroup $H$ of $G$. It is immediate from Definition \ref{defeqGdiag} that $f$ is an equivalence of $G$-diagrams if and only if $f^H\colon X^{H}\to Y^H$ is an equivalence of $I^H$-shaped diagrams of spaces, for every subgroup $H$ of $G$. In this sense, the fixed points diagrams of a $G$-diagram $X$ retain all the homotopical information of $X$.
\end{rem}

\subsection{Equivariant homotopy colimits and fixed points}

Let $G$ be a discrete group, and let $a\colon G\rightarrow Cat$ be a $G$-action on a category $I=a(\ast)$. We study the interaction between the $G$-homotopy colimit of a $G$-diagram and its fixed points diagrams, as defined in \ref{deffixeddiag}.

\begin{prop}\label{resmapF}
Let $X\in \Top^{I}_a$ and $K\in \Top^{I^{op}}_{a^{op}}$ be two $G$-diagrams of spaces. There is a natural homeomorphism
\[(X\otimes_{I} K)^G\cong X^G\otimes_{I^G} K^G\]
In particular for $K=B(-/I)^{op}$ this gives a natural homeomorphism
\[(\hocolim_IX)^G\cong \hocolim_{I^G}X^G\]
\end{prop}

\begin{proof}
Since fixed points commute with colimits, there is a canonical homeomorphism
\[(X\otimes_I K)^G\cong\colim\left((\coprod_{\substack{\alpha\colon i\rightarrow j\\
\in \hom I}}K_i\times X_j)^G\rightrightarrows (\coprod_{i\in \Ob I}K_i\times X_i)^G\right)\]
Fixed points commute with $G$-coproducts, in the sense that the above is homeomorphic to
\[(X\otimes_I K)^G\cong\colim\left(\coprod_{\substack{\alpha\colon i\rightarrow j\\
\in \hom I^G}}(K_i\times X_j)^G\rightrightarrows \coprod_{i\in Ob I^G}(K_i\times X_i)^G\right)\]
Finally, commuting the fixed points and the products we obtain an isomorphism between the equalizer above and $X^G\otimes_{I^G} K^G$.

When $K=B(-/I^{op})$, the formula above gives a homeomorphism
\[(\hocolim_IX)^G\cong X^G\otimes_{I^G} \big(B(-/I^{op})\big)^G\]
For every object $i$ of $I^G$ there are natural homeomorphisms
\[\big(B(i/I^{op})\big)^G\cong B\big((i/I^{op})^G\big)\cong B\big(i/(I^{G})^{op}\big)\]
that identify $X^G\otimes_{I^G} \big(B(-/I^{op})\big)^G$ with $\hocolim_{I^G}X^G$.
\end{proof}

\begin{rem}
The description of the fixed points of the $G$-homotopy colimit in terms of fixed points diagrams given in \ref{resmapF} makes it possible to deduce virtually all the equivariant homotopical properties of the $G$-homotopy colimit functor from the classical homotopical properties of the homotopy colimit. For example, it follows immediately from \ref{resmapF} that the $G$-homotopy colimits of equivalent $G$-diagrams are equivalent as $G$-spaces, recovering a result of \cite[6.1]{vilf}. Another example is the equivariant Thomason's Theorem \ref{GThom} below.
We will see in \S\ref{fixholim} that the relationship between $G$-homotopy limits and fixed points is more involved. 
\end{rem}

An immediate application of Proposition \ref{resmapF} is the equivariant version of Thomason's Theorem. Let $X\colon I\to Cat$ be a $G$-diagram of categories. The Grothendieck construction $I\wr X$ has an induced $G$-action, defined on objects by
\[g\cdot \big(i\in I\ ,\ x\in\Ob X_i\big)=\big(gi\ ,\ gx\in X-{gi}\big)\]
and on morphisms by
\[g\cdot \big( i\stackrel{\alpha}{\to} j\ ,\  \alpha_{\ast}x\stackrel{\gamma}{\to} y\big)=\big( gi\stackrel{g\alpha}{\longrightarrow} gj\ ,\  (g\alpha)_{\ast}(gx)=g(\alpha_\ast x)\stackrel{g\gamma}{\longrightarrow} gy\big)\]
The equality $(g\alpha)_{\ast}(gx)=g(\alpha_\ast x)$ expresses the naturality of $g\colon X\to X\circ g$. Thus the classifying space $B(I\wr X)$ inherits a $G$-action.
The following result was proved in \cite[2.27]{Gdiags} as a special case of a general equivariant Fubini Theorem. The proof presented here is more direct, reducing it to Thomason's Theorem \cite{Thomason} by a fixed points argument.

\begin{cor}[\cite{Gdiags}]\label{GThom}
Let $X\colon I\to Cat$ be a $G$-diagram of small categories, and let $BX$ be the $G$-diagram of spaces obtained by composing with the geometric realization. There is a natural equivalence of $G$-spaces
\[B(I\wr X)\stackrel{\simeq}{\longrightarrow}\hocolim_IBX\]
\end{cor}

\begin{proof}
The map $\eta_X\colon B(I\wr X)\to\hocolim_IBX$ defined in \cite{Thomason} is equivariant. For a subgroup $H$ of $G$, there is a natural isomorphism of categories $(I\wr X)^H\cong I^H\wr X^H$. Under this isomorphism and the homeomorphism of \ref{resmapF}, the map $\eta_{X}^H$ corresponds to the map
\[\eta_{X^H}\colon I^H\wr X^H\longrightarrow \hocolim_{I^H}B(X^H)\]
which is an equivalence by Thomason's Theorem \cite{Thomason}.  
\end{proof}

\subsection{Equivariant homotopy limits and fixed points}\label{fixholim}

The relationship between the fixed points of the $G$-homotopy limit and the homotopy limit of the fixed points diagrams is more involved than it is for $G$-homotopy colimits. Given a pair of $G$-diagrams of spaces  $K,X\in \Top^{I}_a$ there is a restriction map
\[\Hom_I(K,X)^G\longrightarrow \Hom_{I^G}(K^G,X^G)\]
which sends a morphism of $G$-diagrams to its restriction on the fixed points diagrams. This is generally far from being an equivalence.
When $K=B(I/_{-})$ we describe the homotopy fibers of the restriction map $(\holim_{I}X)^G\rightarrow \holim_{I^G}X^G$ in Proposition \ref{resfixespts} below. It turns out that in order to describe the whole fixed points space $\Hom_I(K,X)^G$ one needs to consider the natural transformations between the $H$-fixed points diagrams of $K$ and $X$, for every subgroup $H$ of $G$. This mapping spaces for the various subgroups of $G$ are related by means of the twisted arrow category.

Let us recall from \cite{DK} that the twisted arrow category $Tw(C)$ of a category $C$ has objects the morphisms $f\colon c\to d$ in $C$. A morphisms $f\to f'$ in $Tw(C)$ is a commutative diagram
\[\xymatrix@=13pt{c\ar[r]^-f\ar[d]& d\\
c'\ar[r]_-{f'}&d'\ar[u]}\]
Let  $\mathcal{O}_G$ be the orbit category of $G$. We use the twisted arrow category $Tw(\mathcal{O}^{op}_G)$ to glue together the mapping spaces between the fixed point diagrams of $K$ and $X$. An object of $Tw(\mathcal{O}^{op}_G)$ is an equivariant map $G/H\stackrel{f}{\leftarrow} G/L$. This induces a functor $f^{\ast}\colon I^H\to I^L$, defined on objects by
\[f^{\ast}i=f(L)\cdot i\]
A similar formula defines $f^\ast$ on morphisms. Precomposing a diagram with $f^{\ast}$ leads to a functor $f_!\colon \Top^{I^L}\to \Top^{I^H}$. Given two $G$-diagrams of spaces $K,X\in Top^{I}_{a}$ define a functor
\[Tw(\mathcal{O}^{op}_G)^{op}\longrightarrow \Top\]
by sending an object $G/L\stackrel{f}{\leftarrow} G/H$ of $Tw(\mathcal{O}^{op}_G)$ to the space of natural transformations of $I^H$-diagrams $\Hom_{I^H}(K^H,f_! X^L)$. On morphisms this functor is defined as
\[\vcenter{\hbox{\xymatrix{G/H&\ar[l]_-f G/L\ar[d]_{b}\\
G/H'\ar[u]^a&\ar[l]^-{f'}G/L'}}} \longmapsto
\vcenter{\hbox{\xymatrix@C=20pt{\Hom_{I^H}(K^H,f_! X^L)&&\Hom_{I^H}(a_! K^{H'},a_! f'_! b_{\ast} X^{L'})\ar[ll]_-{(-)\circ a(H')}\\
\Hom_{I^{H'}}(K^{H'},f'_! X^{L'})\ar@{-->}[u]\ar[rr]_-{\res_{a^{\ast}}}
&&\Hom_{I^H}(a_! K^{H'},a_! f'_! X^{L'})\ar[u]_{b(L)\circ (-)}
}}}\]
In the right-hand square, the lower horizontal map restricts a natural transformation along the functor $a^{\ast}\colon I^H\to I^{H'}$. The top horizontal map is precomposition with pointwise multiplication by $a(H')$ in the $G$-structure on $K$, in symbols $a(H')\colon K^{H}_i\to K^{H}_{a(H)\cdot i}= (a_\ast K^H)_i$. Similarly, the right vertical map composes a natural transformation with the action of a representative of $b(L)$ for the $G$-structure of $X$. Explicitly, a natural transformation $\Phi\colon K^{H'}\to f'_\ast X^{L'}$ is sent to
\[\xymatrix@C=40pt{K^{H}_i\ar[r]^-{a(H')}& K^{H'}_{a(H')\cdot i}\ar[r]^-{\Phi_{a(H')i}}&X^{L'}_{f'(L')a(H')i}\ar[r]^-{b(L)}&X^{L}_{b(L)f'(L')a(H')i}=X^{L}_{f(L)i}
}\]

\begin{prop}\label{twisted}
For every pair of $G$-diagrams $K,X$ in $\Top^{I}_{a}$, there is a natural homeomorphism
\[\Hom_I(K,X)^G\cong \lim_{\substack{f\colon G/L\to G/H \\
\in Tw(\mathcal{O}^{op}_G)^{op}}}\Hom_{I^H}(K^H,f_{!}X^L)\]
In particular when $K=B(I/-)$ this is a homeomorphism
\[(\holim_I X)^G\cong \lim_{\substack{f\colon G/L\to G/H \\
\in Tw(\mathcal{O}^{op}_G)^{op}}}
\holim_{I^H}f_{!}X^L\]
\end{prop}

\begin{proof}
In trying to dualize the argument of the proof of \ref{resmapF}, one encounters the problem that $G$-indexed products do not commute with fixed points. Instead, we express the fixed points of the mapping space as a mapping space on a more complicated category. The functor $a\colon G\to Cat$ which defines the $G$-action on $I$ induces a functor $\overline{a}\colon \mathcal{O}^{op}_G\to Cat$ that sends $G/H$ to the fixed points category $I^H$, and a $G$-map $G/H\stackrel{f}{\leftarrow} G/L$ to the functor $f^\ast\colon I^H\to I^L$ described above. Recall the isomorphism of categories $\Top^{I}_{a}\cong \Top^{G\ltimes_a I}$ of \ref{holimGrot}, where $G\ltimes_a I$ is the Grothendieck construction $G\wr a$. The canonical inclusion $G\to \mathcal{O}^{op}_G$ induces a functor $\Top^{\mathcal{O}^{op}_G\wr \overline{a}}\to \Top^{G\wr a}\cong\Top^{I}_{a}$. Theorem \cite[2.28]{Gdiags} shows that this functor is the left adjoint of a Quillen equivalence
\[L\colon \Top^{\mathcal{O}^{op}_G\wr \overline{a}}\rightleftarrows \Top^{G\wr a}\cong\Top^{I}_{a} \colon R\]
where the left-hand category has the projective model structure, and the right-hand category has a model structure where the equivalences are the equivalences of $G$-diagrams of \ref{defeqGdiag}. This is a sort of Elmendorf Theorem for the category of $G$-diagrams. The functor $R$ sends a $G$-diagram $X$ to the diagram with vertices
\[R(X)_{(G/H,i\in I^H)}=X^{H}_i\]
The counit of this adjunction is an isomorphism, giving a homeomorphism
\[\Hom(K,X)^G\cong \Hom_{\mathcal{O}^{op}_G\wr \overline{a}}\big(R(K),R(X)\big)\]
(see \cite[2.28]{Gdiags}). We finish the proof by defining a homeomorphism between the right-hand side and the limit of \ref{twisted}. An element of  $\Hom_{\mathcal{O}^{op}_G\wr \overline{a}}\big(R(K),R(X)\big)$ is the data of a map $\Phi_{(H,i)}\colon K^{H}_i\to X^{H}_i$ for every subgroup $H$ of $G$ and fixed object $i\in I^H$, subject to compatibility conditions corresponding to the morphisms of $\mathcal{O}^{op}_G\wr \overline{a}$. For every equivariant map $G/H\stackrel{f}{\leftarrow} G/L$ define a map
\[\Hom_{\mathcal{O}^{op}_G\wr \overline{a}}\big(R(K),R(X)\big)\longrightarrow \Hom_{I^H}(K^H,f_{!}X^L)\]
by sending a collection of maps $\Phi$ as above to
\[K^{H}_i\stackrel{\Phi_{(H,i)}}{\xrightarrow{\hspace{1cm}}} X^{H}_i\stackrel{f(L)}{\xrightarrow{\hspace{1cm}}} X^{L}_{f(L)\cdot i}=(f_! X^L)_i\] 
Naturality of $\Phi$ insures that these maps are compatible with the morphisms in $Tw(\mathcal{O}^{op}_G)^{op}$, and this defines a map
\[\Hom_{\mathcal{O}^{op}_G\wr \overline{a}}\big(R(K),R(X)\big)\longrightarrow \lim_{f\in Tw(\mathcal{O}^{op}_G)^{op}}\Hom_{I^H}(K^H,f_{!}X^L)\]
The inverse of this map sends a collection of natural transformations $\Psi_f\colon K^H\to f_{\ast}X^L$ in the limit to
\[\Psi_{\id_{G/H}}\colon K^{H}_i{\longrightarrow} X^{H}_i\]
\end{proof}

\begin{rem}
The limit over the Twisted arrow category of \ref{twisted} is in general not a homotopy limit, and hence it is not homotopy invariant. This prevents us from easily deduce homotopical properties of the equivariant mapping space $\Hom_I(K,X)$ from the homotopical properties of the fixed points diagrams. For example, proving that the $G$-homotopy limit functor preserves equivalences requires a considerable amount of work. The equivariant mapping space $\Hom_I(K,X)$ does however enjoy many equivariant homotopical properties, which are studied extensively in \cite{Gdiags} using a model-categorical approach.
\end{rem}

We turn our attention to the connectivity of the restriction map  $(\holim_{I}X)^G\rightarrow \holim_{I^G}X^G$ which will play a key role in the proof of the equivariant Blakers-Massey Theorem \ref{BM}.
Let $\iota\colon I^G\rightarrow I$ be the inclusion of the fixed points category. This is a $G$-equivariant functor for the trivial $G$-action on $I^G$. Therefore  the functor $B(\iota/-)\colon I\to\Top$ has a natural $G$-structure (see Example \ref{overcatGdiag}). There is a pointwise injective morphism $B(\iota/-)\rightarrow B(I/-)$ in $\Top_{a}^I$, and we denote its pointwise quotient $\bigslant{B(I/-)}{B(\iota/-)}$. This is indeed a $G$-diagram since $\Top_{a}^I$ has all colimits. Moreover it is canonically a diagram of based spaces, meaning that there is a canonical natural transformation $\ast\rightarrow \bigslant{B(I/-)}{B(\iota/-)}$.

\begin{prop}\label{resfixespts} Let $X\colon I\to \Top$ be a $G$-diagram of spaces and suppose that the simplicial set $NI/_i$ is finite dimensional for every object $i$ of $I$. Then the restriction map $(\holim_{I}X)^G\rightarrow \holim_{I^G}X^G$ is a fibration. Given a natural transformation $\ast\rightarrow X$ in $\Top_{a}^I$, the fiber of the restriction map over the associated constant natural transformation $B(I^G/-)\rightarrow \ast\stackrel{}{\rightarrow} X^G$ is the fixed points space
\[\Hom_\ast\Big(\bigslant{B(I/-)}{B(\iota/-)},X\Big)^G\]
of the space of natural transformations of pointed maps.
\end{prop}

\begin{proof} Since the $G$-action on $I^G$ is trivial, there is a natural homeomorphism
$\Hom(B(I^G/-),X^G)\cong \Hom(B(I^G/-),X|_{I^G})^G$ and the restriction map fits into a commutative diagram
\[\xymatrix@C=50pt{\Hom_I\big(B(I/-),X\big)^G\ar[r]^-{res}\ar[dr]&
\Hom_{I^G}\big(B(I^G/-),X|_{I^G}\big)^G\\
&\Hom_I\big(B(\iota/-),X\big)^G\ar[u]_{\cong}
}\]
The vertical map is the restriction on $G$-fixed points of an adjunction isomorphism of the kind $\Hom_D\big(B(F/-),Z\big)\cong \Hom_C\big(B(C/-),F^{\ast}Z\big)$ for functors $C\stackrel{F}{\rightarrow }D\stackrel{Z}{\rightarrow } \Top$ (see e.g. \cite[19.6.6]{hirsch}). We show that the diagonal map $\Hom_I\big(B(I/-),X\big)\rightarrow
\Hom_I\big(B(\iota/-),X\big)$ induced by the inclusion $B(\iota/-)\rightarrow B(I/-)$ is a fibration in $\Top^G$, and that the point-fiber is the space of natural transformations of \ref{resfixespts}.

The diagonal map is a fibration provided $B(\iota/-)\rightarrow B(I/-)$ is a cofibration in the $G$-projective model structure of $\Top_{a}^I$, as defined in \cite[2.6]{Gdiags}. The following argument is analogous to \cite[18.4.1]{hirsch}.
The map $B(\iota/-)\rightarrow B(I/-)$ is a cofibration if we can solve the lifting problem
\[\xymatrix@=20pt{B(\iota/-)\ar[r]\ar[d]&E\ar@{->>}[d]^{\simeq}\\
B(I/-)\ar[r]\ar@{-->}[ur]_-l&B
}\]
for every acyclic fibration $E\rightarrow B$ in $\Top_{a}^I$. This is a morphism of $G$-diagrams with the property that $E_i\rightarrow B_i$ is an acyclic fibration of $G_i$-spaces for every object $i$ of $I$. For every $G$-diagram $X$ in $\Top_{a}^I$ and every object $i$ of $I$ let $L_iX$ be the $G_i$-space
\[L_iX=\colim\limits_{(j\rightarrow i)\neq \id_i}X_j\]
If we can prove that the relative latching maps 
\[B(\iota/i)\coprod\limits_{L_iB(\iota/-)}
L_iB(I/-)\longrightarrow B(I/i)\]
are cofibrations of $G_i$-spaces for every $i$, one can build a lift $l$ inductively on the filtration of $I$ defined by the degree function $\deg=\dim N(I/-)\colon Ob I\rightarrow\mathbb{N}$. The construction of $l$  is completely analogous to \cite[A.6]{Gdiags}. Since the geometric realization functor preserves cofibrations and it commutes with colimits, it is enough to prove that
\[N(\iota/i)\coprod\limits_{L_iN(\iota/-)}
L_iN(I/-)\longrightarrow N(I/i)\]
is a cofibration of simplicial $G_i$-sets. A cofibration of simplicial $G_i$-sets is a $G_i$-equivariant map that is levelwise injective,
%\cite{Shipley notes}
that is a $G_i$-equivariant maps which is a cofibrations of underlying simplicial sets.
The maps $L_iN(I/-)\rightarrow N(I/i)$ and $L_iN(\iota/-)\rightarrow N(\iota/i)$ are cofibrations of simplicial sets as both $N(I/-)$ and $N(\iota/-)$ are (Reedy) cofibrant diagrams of simplicial sets (see e.g. \cite[14.8.5]{hirsch}).
It follows that the $G_i$-equivariant map from the pushout
\[\xymatrix@R=10pt@C=50pt{L_iN(\iota/-)\ \ar[d]\ar@{>->}[r]&N(\iota/i)\ar[d]\ar@/^1pc/[ddr]\\
L_iN(I/-)\ \ar@{>->}[r]\ar@{>->}@/_1pc/[drr]&
N(\iota/i)\coprod\limits_{L_iN(\iota/-)}\ar@{-->}[dr]
L_iN(I/-)\\
&& N(I/i)
}\]
is also a cofibration of simplicial sets.
%so are $L_iN(I/-)\rightarrow N(I/i)$ and $L_iN(\iota/-)\rightarrow N(\iota/-)\rightarrow N(I/-)$
%as both $N(I/-)$ and $N(\iota/-)$ are (Reedy) cofibrant (see \cite[14.8.5]{hirsch}). Thus it is enough to show that if an element $\big((i_0\rightarrow\dots\rightarrow i_n)\in N_nI^G,(\alpha\colon i_n\rightarrow i)\in \hom I \big)$ of $N_n\iota/_i$ and a pair $\big((\delta\colon j\rightarrow i)\neq\id_i\in \hom I, (j_0\rightarrow\dots\rightarrow j_n\stackrel{\beta}{\rightarrow} j)\in N_nI/_j\big)$ representing an element of $L_iN_nI/_{(-)}$ are sent to the same element of $NI/_i$, then they must come from $L_iN_n\iota/_{(-)}$. If they are sent to the same element of $NI/_i$ we must have
%\[(i_0\rightarrow\dots\rightarrow i_n)=(j_0\rightarrow\dots\rightarrow j_n)  \ \ \ \ \mbox{and}\ \ \ \ \alpha=\delta\circ\beta\]
%Therefore $\big(\delta\colon j\rightarrow i,(i_0\rightarrow\dots\rightarrow i_n),\beta\colon i_n\rightarrow j\big)$ represent an element of $L_iN_n\iota/_{(-)}$ that maps to both $\big((i_0\rightarrow\dots\rightarrow i_n),\alpha)$ and $\big(\delta, (j_0\rightarrow\dots\rightarrow j_n\stackrel{\beta}{\rightarrow} j)\big)$. 
This finishes the proof that the restriction map is a fibration.

We still need to describe the fiber of the restriction map. Since fibers commute with fixed points, equalizers and products, the fiber of the restriction map $\Hom_I(B(I/-),X)^G\rightarrow\Hom_I(B(\iota/-),X)^G$ over $ B(\iota/-)\rightarrow\ast\stackrel{x}{\rightarrow} X$ is the equalizer
\[\lim\bigg(\prod_{i\in \Ob I}\fib\left(
\vcenter{\hbox{\xymatrix{\map(B(I/i), X_i)\ar[d]\\ \map(B(\iota/i), X_i)}}}
\right)\rightrightarrows \prod_{\substack{\alpha\colon i\rightarrow j\\
\in \hom I}}\fib\left(
\vcenter{\hbox{\xymatrix{\map(B(I/i), X_j)\ar[d]\\ \map(B(\iota/i), X_j)}}}
\right)\bigg)^G\]
where the fibers are taken over the constant maps 
$B(\iota/i)\rightarrow\ast\stackrel{x_i}{\rightarrow} X_i$ and $B(\iota/i)\rightarrow\ast\stackrel{x_j}{\rightarrow} X_j$ respectively. These fibers are naturally isomorphic to the spaces of pointed maps from the quotients, giving a homeomorphism between the equalizer above and
\[\lim\bigg(\prod_{i\in \Ob I}\map_\ast\left(\bigslant{B(I/i)}{B(\iota/i)},X_i\right)\rightrightarrows \prod_{\substack{\alpha\colon i\rightarrow j\\
\in \hom I}}\map_\ast\left(\bigslant{B(I/i)}{B(\iota/i)},X_j\right)\bigg)^G\]
This is precisely the mapping space of the statement.
\end{proof}

\begin{cor}\label{connresmap}
Let $X\colon I\to \Top$ be a $G$-diagram of spaces. Suppose that the simplicial sets $NI/_i$ are finite dimensional, and that the composition of the canonical maps \[(\lim_{I}X)^G\to(\holim_{I}X)^G\to
\holim_{I^G}X^G\]
is surjective in $\pi_0$. Then the restriction map $(\holim_{I}X)^G\rightarrow \holim_{I^G}X^G$ is at least
\[\min_{i\in Ob I}\min_{\substack{H\leq G_i\\ \hom (\iota^H/{i})\subsetneqq \hom (I^H/{i})}
}\left(\conn X_{i}^H-\dim N(I^H/i)\right)+1\]
connected, where $\iota^H\colon I^G\rightarrow I^H$ is the inclusion.
\end{cor}
\begin{rem}
The $\pi_0$ hypothesis of Corollary \ref{connresmap} is satisfied if for every fixed object $i$ of $I^G$ the connectivity of $X^G_{i}$ is greater or equal to the dimension of $NI/_{i}^G$, since in this case $\holim_{I^G}X^G$ is connected (see \ref{connfixedptsholim}).
%\item A subgroup $H\leq G_i$ for which $\hom \iota^H/i\subsetneqq \hom I^H/i$ also satisfies $\hom I^G\subsetneqq \hom I^H$. One can simplify the formula of the corollary slightly, by replacing it with the smaller connectivity range
%\[\min_{i\in Ob I}\min_{\substack{H\leq G_i\\ \hom I^G\subsetneqq \hom I^H
%}}\left(\conn X_{i}^H-\dim NI^H/{i}\right)+1\]
%\end{enumerate}
\end{rem}

\begin{proof}[Proof of\ref{connresmap}]
Let us write $K$ for the quotient of $B(I/-)$ by the subdiagram $B(\iota/-)$. It is cofibrant as the inclusion is a cofibration (cf. \ref{resfixespts}). The $\pi_0$ assumption insures that all the homotopy fibers are a space of natural transformation from $K$ to $X$ for some basepoint $\ast\rightarrow X$, as in \ref{resfixespts}. We show in \ref{connfixedptsholim} that this space of natural transformations is at least
\[\min_{i\in Ob I}\min_{\substack{H\leq G_i\\ K^{H}_i\neq\ast}}\left(\conn X_{i}^H-\dim K_{i}^H\right)\]
connected. The space $K^{H}_i$ is a point precisely when $B(I^H/{i})=B(\iota^H/{i})$. This is the case if the equality is verified in nerve level one, that is when $\hom (I^H/{i})=\hom(\iota^H/i)$.
\end{proof}

\section{The equivariant Blakers-Massey Theorem and applications}

\subsection{The equivariant Blakers-Massey Theorem}\label{secBM}

Let $G$ be a discrete group and let $J$ be a finite $G$-set. We write $\mathcal{P}(J)$ for the poset category of subsets of $J$ ordered by inclusion. The $G$-action on $J$ induces a $G$-action on the category $\mathcal{P}(J)$, by sending a subset $U$ of $J$ to its image $g(U)$ under the map $g\colon J\to J$. This action restricts to the subposets 
\[\mathcal{P}_0(J)=\mathcal{P}(J)\backslash\emptyset\ \ \ \ \ \ \ \ \ \ \ \ \ \ \ \ \mathcal{P}_1(J)=\mathcal{P}(J)\backslash J\]
A $J$-cube of spaces is a $G$-diagram $X\colon \mathcal{P}(J)\to\Top$. The initial and final vertices $X_\emptyset$ and $X_J$ have a $G$-action induced by the $G$-structure of $X$, since $\emptyset$ and $J$ are fixed objects of $\mathcal{P}(J)$. Moreover the canonical maps
\[X_{\emptyset}\rightarrow \holim_{\mathcal{P}_0(J)}X
\ \ \ \ \ \ \ \ \ \ \ \ \ \ \ \
\hocolim_{\mathcal{P}_1(J)}X\rightarrow X_{J}\]
are $G$-equivariant with respect on the $G$-actions on the homotopy limit and on the homotopy colimit of \S\ref{secGdiags}.

\begin{defn}
Let $J$ be a finite $G$-set and let $X\in\Top^{\mathcal{P}(J)}_a$ be a $J$-cube. Given a function $\nu\colon \{H\leq G\}\rightarrow \mathbb{N}$ which is invariant on conjugacy classes, we say that $X$ is $\nu$-cartesian if for every subgroup $H$ of $G$ the map of spaces
\[X_{\emptyset}^H\longrightarrow (\holim_{\mathcal{P}_0(J)}X)^H\]
is $\nu(H)$-connected. We say that $X$ is homotopy cartesian if it is $\nu$-cartesian for every function $\nu$. Dually, we say that $X$ is $\nu$-cocartesian if
\[(\hocolim_{\mathcal{P}_1(J)}X)^H\longrightarrow X_{J}^H\]
is $\nu(H)$-connected for every subgroup $H$ of $G$, and that $X$ is homotopy cocartesian if it is $\nu$-cocartesian for every $\nu$.
\end{defn}

\begin{rem}
There is a natural isomorphism of categories $\mathcal{P}(J)^H\cong \mathcal{P}(J/H)$. Under this isomorphism the fixed points diagram $X^H\colon \mathcal{P}(J)^H\to \Top$ of a $J$-cube $X$ is a $J/H$-cube. By the fixed points description of the $G$-homotopy colimit of Proposition \ref{resmapF}, a $J$-cube $X$ is $\nu$-cocartesian precisely when the $J/H$-cube $X^H$ is $\nu(H)$-cocartesian for every subgroup $H$ of $G$. Because of the failure of $G$-homotopy limit to commute with fixed points, the analogous statement does not hold for cartesian $J$-cubes (cf. \ref{resfixespts}).
\end{rem}

Given a finite $G$-set $J$ and a subgroup $H$ of $G$ we denote by $\Pa_H(J)$ the set of partitions of $J$ by $H$-subsets. This is the set of coverings $\{\emptyset\neq T_\alpha\subset J\}_{\alpha}$ of $J$ by non-empty subsets to which the $H$-action on $J$ restricts, and such that $T_\alpha$ and $T_{\alpha'}$ are disjoint for $\alpha\neq \alpha'$. For a subset $U\subset J$ let $\Eff_H(U)$ be the set of subgroups $L\leq H$ ``that act effectively on $U$'', defined by
\[\Eff_H(U)=\left\{\begin{array}{ll}
\{L\leq H_U\}&\mbox{if } H_U\neq H\\
\\
\{L\leq H \ |\ U/L \neq U/H\}&\mbox{if } H_U= H
\end{array}\right.\]
Here $H_U$ denotes the stabilizer group in $H$ of the object $U$ of $\mathcal{P}(J)$. This is the largest subgroup of $H$ whose action on $J$ restricts to $U$.
The important point of this definition is that the elements of $\Eff_H(U)$ are all proper subgroups of $H$.

\begin{theorem}\label{BM}
Let $X\in Top^{\mathcal{P}(J)}_a$ be a $J$-cube, and suppose that for every subgroup $H$ of $G$ and every non-empty $H$-subset $U\subset J$ the restriction $X|_{\mathcal{P}(U)}$ is $\nu^U$-cocartesian for some function $\nu^U\colon \{K\leq H\}\rightarrow \mathbb{Z}$. Suppose moreover that these functions satisfy $\nu^U\leq \nu^V$ whenever $U\subset V$ are non-empty $H$-subsets of $J$. Then $X$ is $\nu$-cartesian, where $\nu$ is the function
\[\nu(H)=\!\min\left\{\!\min_{\{T_\alpha\}\in\Pa_H(J)}\{
\sum_{\alpha}\nu^{T_\alpha}(H)\}-|J/H|+1\ ,\ \min_{\emptyset\neq U\subset J}\min_{L\in \Eff_H(U)}\big\{\conn X^{L}_U-|U/L|+1\big\}
\right\}\]
\end{theorem}

\begin{rem}
The first term of the minimum is analogous to the formula of Goodwillie's Blakers-Massey Theorem \cite[2.5]{calcII} for $n$-cubes. The second term is purely equivariant: if the $G$-action on $J$ is trivial the set of subgroups $\Eff_H(U)$ is empty. If $J$ is the set with $n$ elements and trivial $G$-action, this is the standard Blakers-Massey Theorem for $n$-cubes of $G$-spaces.  
\end{rem}

The Blakers-Massey Theorem \ref{BM} has the following dual statement.

\begin{theorem}\label{dualBM}
Let $X\in Top^{\mathcal{P}(J)}_a$ be a $J$-cube, and suppose that for every subgroup $H$ of $G$ and for every non-empty $H$-subset $U\subset J$ the cube $X|_{\mathcal{P}(U)}$ is $\nu^U$-cartesian. Suppose moreover that the functions $\nu^U\colon \{K\leq H\}\rightarrow \mathbb{N}$ satisfy $\nu^U\leq \nu^V$ whenever $U\subset V$ are non-empty $H$-subsets of $J$. Then $X$ is $\nu$-cocartesian, for
\[\nu(H)\!=\!\!\!\min_{\{T_\alpha\}\in\Pa_H(J)}\left\{
|J/H|-1+\sum_{\alpha}\min\Big\{\nu^{T_\alpha}(H)\ , \min_{\emptyset\neq U\subset T_\alpha}\min_{L\in \Eff_H(U)}\{\conn X^{L}_U-|U/L|+1\}\Big\}
\!\right\}\]
\end{theorem}

\begin{proof}[Proof of \ref{BM}]
For each subgroup $H$ of $G$, we need to calculate the connectivity of the canonical map $\phi\colon X^{H}_\emptyset\rightarrow (\holim_{\mathcal{P}_0(J)} X)^H$ at any basepoint in $X^{H}_\emptyset$. Since the empty set is initial in $\mathcal{P}(J)^H$, a basepoint of $X^{H}_\emptyset$ defines a basepoint $\ast\rightarrow X^H$ for the whole fixed points diagram. There is a commutative diagram
\[\xymatrix@R=15pt@C=40pt{X^{H}_\emptyset\ar[r]^-\phi\ar[dr]_-\psi&(
\displaystyle\holim_{\mathcal{P}_0(J)} X)^H\ar[d]^-r\\
&\displaystyle\holim_{\mathcal{P}_0(J)^H} X^H
}
\]
where $r$ is the restriction map of Proposition \ref{resfixespts}, and $\psi$ is the canonical map for the fixed points diagram $X^H$.
Our map is as connected as
\[\conn\phi=\min\{\conn \psi,\conn r-1\}\]
The connectivity of $\psi$ expresses how cartesian the $J/H$-cube $X^H$ is, and it is determined by the standard Blakers-Massey Theorem of \cite[2.5]{calcII}. In order to apply this theorem we need to compare the cocartesianity of the subcubes of $X^H$. Under the canonical isomorphism $\mathcal{P}(J/H)=\mathcal{P}(J)^H$ a non-empty subset $U\subset J/H$ corresponds to a non-empty $H$-subsets of $J$. By assumption the cube $X^H$ is $\nu^U(H)$-cocartesian, and for  subsets $U\subset V \subset J/H$ the inequality $\nu^U(H)\leq \nu^V(H)$ holds. By \cite[2.5]{calcII} the cube $X^H$ is
\[\min_{\{T_\alpha\}}\{
\sum_{\alpha}\nu^{T_\alpha}(H)\}-|J/H|+1\]
cartesian,
where $\{T_\alpha\}$ runs over the partitions of $J/H$, which correspond to the partitions of $J$ by $H$-subsets. This shows that $\psi$ is as connected as the first term of the minimum of the statement.

Now we calculate the connectivity of the restriction map $r$.
We can assume that $\psi$ is at least $0$-connected, otherwise the connectivity range of \ref{BM} gives us no information about the cartesianity of $X$. The $\pi_0$-hypothesis of \ref{connresmap} is satisfied, since there is a commutative diagram
\[\xymatrix@R=15pt{\pi_0X^{H}_\emptyset\ar[d]
\ar[drr]^-{\pi_0\psi}
\\
\displaystyle\pi_0(\lim_{\mathcal{P}_0(J)} X)^H\ar[r]&\displaystyle\pi_0(\holim_{\mathcal{P}_0(J)} X)^H\ar[r]_-r
&\displaystyle\pi_0\holim_{\mathcal{P}_0(J)^H} X^H
}
\]
and the surjectivity of $\pi_0\psi$ implies the surjectivity of the lower composite. By \ref{connresmap} the connectivity of $r$ is
\[\min_{\emptyset\neq U\subset J}\min_{\substack{L\leq H_U\\ \hom (\iota^L/U)\subsetneqq \hom \mathcal{P}(U)^L}
}\left(\conn X_{U}^L-|U/L|+1\right)+1\]
where $\iota^L\colon \mathcal{P}(J)^H\rightarrow  \mathcal{P}(J)^L$ is the inclusion. It remains to show that if a subgroup $L$ of $H_U$ satisfies the condition $\hom (\iota^L/U)\subsetneqq \hom \mathcal{P}(U)^L$ , then $L$ belongs to $\Eff_H(U)$ (in fact, the two conditions are equivalent). If $H_U$ is different from $H$ clearly $L$ belongs to $\Eff_H(U)$ since it is a proper subgroup of $H$.
If $H_U=H$, the category $\iota^L/U$ is the full subcategory of $\mathcal{P}(J)^H$ of objects that are subsets of $U$. This is the same as the full subcategory $\mathcal{P}(U)^H$ of $\mathcal{P}(U)^L$. These are different precisely when $\mathcal{P}(U/H)\neq \mathcal{P}(U/L)$, that is when $U/H\neq U/L$.
\end{proof}

\begin{proof}[Proof of \ref{dualBM}]
For every subgroup $H$ of $G$ there is a commutative diagram 
\[\xymatrix@R=18pt@C=40pt{\displaystyle(\hocolim_{\mathcal{P}_1(J)}X)^H\ar[dr]\ar[r]^-{\cong}& \displaystyle\hocolim_{\mathcal{P}_1(J)^H}X^H\ar[d]\\
&X^{H}_J}\]
where the homeomorphism is from Proposition \ref{resmapF}. Since $\mathcal{P}_1(J)^H$ is isomorphic to $\mathcal{P}_1(J/H)$, we need to determine how cocartesian the $J/H$-cube $X^H$ is. By the dual Blakers-Massey Theorem \cite[2.6]{calcII}, it is
\[\min_{\{T_\alpha\}\in\Pa_H(J)}\left\{
|J/H|-1+\sum_{\alpha}\omega^{T_\alpha}(H)\right\}\]
cocartesian, if each restriction $X^H|_{\mathcal{P}(T_\alpha/H)}$
is $\omega^{T_\alpha}(H)$-cartesian. To determine $\omega^{T_\alpha}(H)$, consider the commutative diagram
\[\xymatrix@R=18pt@C=40pt{
X_{\emptyset}^H\ar[dr]\ar[r]&\displaystyle(\holim_{\mathcal{P}_0(T_\alpha)}X)^H\ar[d]^-{r}\\
&\displaystyle\holim_{\mathcal{P}(T_\alpha/H)}X^H
}\]
where the vertical map is the restriction. By hypothesis the horizontal map is $\nu^{T_\alpha}(H)$-connected. The connectivity of $r$ is calculated using \ref{connresmap} just as we did in the proof of \ref{BM}, showing that
\[\omega^{T_\alpha}(H)=\min\Big\{\nu^{T_\alpha}(H)\ , \ \min_{\emptyset\neq U\subset T_\alpha}\min_{L\in \Eff_H(U)}\{\conn X^{L}_U-|U/L|+1\}\Big\}\]
\end{proof}

\subsection{The equivariant Freudenthal suspension Theorem}\label{secsusp}

Let $G$ be a discrete group and let $J$ be a finite $G$-set. We let $S^J$ be the permutation representation sphere of $J$, defined as the one-point compactification $\mathbb{R}[J]^+$. Given a pointed $G$-space $X$, we define its $J$-loop and $J$-suspension respectively as the space of pointed maps and the smash product
\[\Omega^J X=\map_\ast(S^J,X)\ \ \ \ \ \ \ \ \ \ \ \ \ \ \Sigma^JX=X\wedge S^J\]
with $G$-action by conjugation on $\Omega^ JX$ and diagonal on $\Sigma^ JX$. The pair of functors $\Sigma^J\colon \Top^{G}_\ast\leftrightarrows\Top^{G}_\ast\colon \Omega^J$ forms an adjunction. The equivariant Freudenthal suspension Theorem (\cite{Nam}, see also \cite{Lewis},\cite{Adams}) calculates the connectivity of the unit of this adjunction
\[X^G\longrightarrow (\Omega^J\Sigma^ JX)^G\]
restricted to the $G$-fixed points spaces. We show that his theorem is a direct consequence of the equivariant Blakers-Massey Theorem \ref{BM}. The relationship between the Blakers-Massey theorem and the Freudenthal suspension Theorem is a well-established fact when $J$ has the trivial $G$-action.

\begin{cor}[\cite{Nam}]\label{susp} Let $X$ be a pointed $G$-space. The unit of the $(\Sigma^{J},\Omega^{J})$-adjunction $\eta\colon X^G\to (\Omega^J\Sigma^{J}X)^G$ is
\[\min\left\{2\conn X^G+1,\min_{\substack{H\leq G\\ J/H\neq J/G}}\conn X^H\right\}\]
connected.
\end{cor}

\begin{proof}
The idea of the proof is to construct the loop space $\Omega^J\Sigma^JX$ as the $G$-homotopy limit of an equivariant cube. Let $J_+$ be the $G$-set $J$ with an added fixed basepoint.
Define a $J_+$-cube $\sigma^JX\colon \mathcal{P}(J_+)\to \Top_\ast$ with vertices \[(\sigma^JX)_U=\left\{\begin{array}{lll} X &, U=\emptyset\\
C^UX &, \emptyset\neq U\subsetneqq J_+\\
\Sigma^JX&, U=J_+
\end{array}\right.\]
where $C^UX$ is the $U$-fold reduced cone of $X$, defined as
\[C^UX=\hocolim_{\mathcal{P}(U)}\big(V\longmapsto\left\{\begin{array}{cl}X&, V=\emptyset\\
\ast&,\mbox{otherwise}
\end{array}\right.\big)\]
This homotopy colimit is performed in the category of pointed spaces. The functor $\sigma^JX$ has a natural $G$-structure. It is defined at the initial and final vertices by the $G$-actions on $X$ and $\Sigma^JX$ respectively, and at the other vertices by the canonical isomorphism
\[C^UX\longrightarrow C^{gU}X\]
induced by the functor $g\colon \mathcal{P}(U)\to\mathcal{P}(gU)$. For the group $G=\mathbb{Z}/2$ and the $G$-set $J=\mathbb{Z}/2$ this is the cube
\[\xymatrix@=12pt{X\ar[rr]\ar[dd]\ar[dr]
&&CX\ar@{<-->}[dl]\ar[dr]\ar@{|}[d]\\
&CX\ar[dd]\ar[rr]&\ar[d]&C^2X\ar[dd]\\
CX\ar@{-}[r]\ar[dr]&\ar[r]&C^2X\ar@{<-->}[dl]\ar[dr]\\
& C^2X\ar[rr]&&\Sigma^{2,1}X
}\]
The dashed maps denote the $G$-structure between the non-fixed vertices, and $\Sigma^{2,1}$ is the suspension by the regular representation of $\mathbb{Z}/2$.
The space $C^UX$ is $G_U$-contractible, since for every subgroup $H$ of $G_U$ Proposition \ref{resmapF} provides a homeomorphism
\[(C^UX)^H\cong C^{U/H}(X^H)\simeq\ast\]
This shows that the restriction of $\sigma^JX$ to $\mathcal{P}_0(J_+)$ is equivalent to the $G$-diagram $\omega^JX\colon\mathcal{P}_0(J_+)\to\Top$ defined by
\[(\omega^JX)_U=\left\{\begin{array}{ll} \ast&, U\neq J_+\\
\Sigma^JX &, U=J_+
\end{array}\right.\]
via the obvious pointed map $\omega^JX\to\sigma^JX|_{\mathcal{P}_0(J_+)}$.
The homotopy limit of $\omega^JX$ is $G$-homeomorphic to $\Omega^J\Sigma^JX$, by inspection on the definition of the Bousfield-Kan formula. This gives a commutative diagram of pointed $G$-spaces
\[\xymatrix@C=8pt{X=(\sigma^JX)_\emptyset\ar[rrrrr]\ar[drrrr]_\eta
&&&&&
\displaystyle\holim_{\mathcal{P}_0(J_+)}\sigma^JX\\
&&&&\Omega^J(\Sigma^JX)\ar@<-1ex>@{}[r]^-\cong& \displaystyle\holim_{\mathcal{P}_0(J_+)}\omega^JX\ar[u]_{\simeq}}\]
where the vertical map is a $G$-equivalence by homotopy invariance of $G$-homotopy limits. The unit map $\eta$ is then as connected as $\sigma^JX$ is cartesian.

We show that $\sigma^JX$ satisfies the condition of the Blakers-Massey Theorem \ref{BM}, and we show that its cartesianity estimate is the same as the range claimed in \ref{susp}.
For each $G$-subset $U$ of $J_+$, we need to find an estimate $\nu^U$ for the cocartesianity of the restriction $\sigma^JX|_{\mathcal{P}(U)}$. For every subgroup $H$ of $G_U$, there are isomorphisms
\[(\hocolim_{\mathcal{P}_1(U)}\sigma^JX)^H\cong \hocolim_{\mathcal{P}_1(U/H)}(\sigma^JX)^H\cong \hocolim_{\mathcal{P}_1(U/H)}\sigma^{J/H}X^H\]
by the fixed points description of \ref{resmapF}.
Since all the vertices of $\sigma^{J/H}X^H|_{\mathcal{P}_1(U/H)}$ are contractible except for the initial one, there is a natural equivalence 
\[(\hocolim_{\mathcal{P}_1(U)}\sigma^{J}X)^H\cong
\hocolim_{\mathcal{P}_1(U/H)}\sigma^{J/H}X^H\stackrel{\simeq}{\longrightarrow} \Sigma^{|U/H|-1}X^H\]
Hence the canonical map $(\hocolim_{{\mathcal{P}_1(U)}} \sigma^JX)^H\to (\sigma^JX)^{H}_U$ factors through the equivalence
\[(\hocolim_{\mathcal{P}_1(U/H)}\sigma^{J}X)^H\stackrel{\simeq}{\longrightarrow}\Sigma^{|U/H|-1}X^H\longrightarrow (\sigma^{J}X)_{U}^H=\left\{\begin{array}{ll}
C^{U/H}X^H&, U\neq J_+\\
\Sigma^{|J/H|}X^H &,U= J_+
\end{array}\right.\]   
This map is $(\conn X^H+|U/H|)$-connected for $U\neq J_+$ and it is an equivalence for $U=J_+$. It follows that $\sigma^JX|_{\mathcal{P}(U)}$ is $\nu^U$-cocartesian for the function
\[\nu^U(H)=\left\{\begin{array}{ll}
\conn X^H+|U/H|&, U\neq J_+\\
\infty &,U= J_+
\end{array}\right.\]
These functions satisfy $\nu^U\leq \nu^V$ for $H$-subsets $U\subset V$ of $J_+$, and the equivariant Blakers-Massey Theorem \ref{BM} applies. The first term of the minimum in the range of \ref{BM} for the group $H=G$ is
\[\min_{\{T_\alpha\}\in\Pa_G(J_+)}\{
\sum_{\alpha}\nu^{T_\alpha}(G)\}-|J_+/G|+1=\min_{\{J_+\}\neq\{T_\alpha\}_{\alpha\in A}}\{
\sum_{\alpha}(|T_\alpha/G|)+|A|\conn X^G\}-|J_+/G|+1\]
The trivial partition $\{J_+\}$ is removed from the minimum because $\nu^{J_+}=\infty$. For any partition $\{T_\alpha\}$ of $J_+$ by $G$-sets, the quotient $J_+/G$ decomposes as the disjoint union of the quotients $T_\alpha/G$. Therefore the sum in the formula above is equal to $|J_+/G|$. The minimum is thus realized when the size of the partition $|A|$ is minimal. This is the partition with two elements $\{J,+\}$, and the quantity above is
\[\min_{\{J_+\}\neq\{T_\alpha\}_{\alpha\in A}}\{|J_+/G|
+|A|\conn X^G\}-|J_+/G|+1=2\conn X^G+1\]
The second term of the minimum in the formula of Theorem \ref{BM} is 
\[\min_{\emptyset\neq U\subset J_+}\min_{H\in \Eff_G(U)}\conn (\sigma^JX)^{H}_U-|J_+/H|+1\]
All the vertices of $\sigma^JX$ have contractible fixed points except for the initial and the final one. The outer minimum is then realized when $U=J_+$, with value
\[\min_{H\in \Eff_G(J_+)}\conn (S^J\wedge X)^{H}-|J_+/H|+1=\min_{H\in \Eff_G(J_+)}\conn X^H\]
The set $\Eff_G(J_+)$ is by definition the set of subgroups $H$ of $G$ for which $J_+/H\neq J_+/G$, which is the same condition as $J/H\neq J/G$.
\end{proof}

\subsection{Equivariant intersections of submanifolds and configuration spaces}\label{secconf}

Let $M$ be a manifold and let $P_1,\dots,P_n$ be a collection of submanifold of $M$. For every integer $2\leq k\leq n$, we define inductively what it means for the collection $P_1,\dots,P_n$ to be $k$-transverse in $M$. Let us denote $\underline{n}=\{1,\dots n\}$.
\begin{enumerate}[label=\roman*)]
\item $P_1,\dots,P_n$ is $2$-transverse if $P_i$ intersects $P_j$ transversally for every $i\neq j$.
\item $P_1,\dots,P_n$ is $k$-transverse if it is $(k-1)$-transverse and if for every set $S\subset \underline{n}$ with $k$ elements and every $s\in S$ the submanifold $\bigcap_{t\in S\backslash s}P_t$ intersects $P_s$ transversally.
\end{enumerate}
Let $m$ be the dimension of $M$, and let $d_i$ be the dimension of $P_i$.
If the collection of submanifolds $P_1,\dots,P_n$ is $n$-transverse, the intersection $\bigcap_{s\in S}P_s$ is either a submanifold of $M$ of codimension $\sum_{s\in S}(m-d_s)$, or it is the empty submanifold.

Now suppose that $G$ is a discrete group, let $J$ be a finite $G$-set, and suppose that $G$ acts properly on a manifold $M$. Let $\{P_j\}_{j\in J}$ be a set of submanifolds indexed on $J$, and suppose that for every group element $g$ the corresponding automorphism of $M$ restricts to a map $g\colon P_j\to P_{gj}$. The $G$-action on $M$ defines a $G$-structure on the cube $M\backslash P_{\bullet}\colon\mathcal{P}(J)\to \Top$ which sends a subset $U\subset J$ to the complement $M\backslash\bigcup_{j\notin U}P_j$. The $G$-structure is defined by the maps
\[g\colon M\backslash\bigcup_{j\notin U}P_j\longrightarrow M\backslash \bigcup_{j\notin U}g(P_j)=M\backslash\bigcup_{j\notin U}P_{gj}=M\backslash\bigcup_{j\notin gU}P_{j}\]

\begin{theorem}\label{intsubman}
Let $M$ be a manifold with proper $G$-action, and let $\{P_j\}_{j\in J}$ be a set of closed submanifolds as above. Suppose moreover that for every $j\in J$ and every subgroup $H$ of $G$ the intersection $P_j\cap M^H$ is a submanifold of $M$, and that the collection of submanifolds $\{P_j\cap M^H\}_{j\in J}$ is $|J|$-transverse in $M^H$. Then
the $J$-cube $M\backslash P_{\bullet}\colon\mathcal{P}(J)\to \Top$ is $\nu$-cartesian, where $\nu$ is the function
\[\min\Big\{\sum_{j\in J}\big(m_H-d_j(H)\big)-2|J/H|+1\ ,
\min_{\substack{
L\leq H\\
J/L\neq J/H
}}\big\{\min\{\conn M^L+1,m_L-\max_{j\in J}d_j(L)\}-|J/L|\big\}
 \Big\}\]
Here $d_j(H)$ is the dimension of $P_j\cap M^H$, and $m_H$ is the dimension of $M^H$.
\end{theorem}

\begin{proof}
In order to understand how cartesian $M\backslash P_{\bullet}$ is it is sufficient, by the Blakers-Massey theorem, to understand how cocartesian the subcubes of $M\backslash P_{\bullet}$ are. Let $H$ be a subgroup of $G$ and let $U$ be an $H$-invariant subset of $J$. The homotopy colimit of the restriction of $M\backslash P_{\bullet}$ to $\mathcal{P}_1(U)$ is $H$-equivalent to
\[\hocolim_{\mathcal{P}_1(U)}M\backslash P_{\bullet}\stackrel{\simeq}{\longrightarrow}\colim_{\mathcal{P}_1(U)}M\backslash P_{\bullet}
=\big(M\backslash\bigcup_{j\notin U}P_j\big)\backslash\bigcap_{u\in U}P_u\]
Hence the cube $(M\backslash P_{\bullet})|_{\mathcal{P}(U)}^{H}$ is $\nu^U(H)$-cocartesian, where $\nu^U(H)$ is the connectivity of the inclusion 
\[\nu^U(H)=\conn\Big(
\big(M\backslash\bigcup_{j\notin U}P_j\big)\backslash\bigcap_{u\in U}P_u\longrightarrow M\backslash\bigcup_{j\notin U}P_j
\Big)^H\]
of $H$-fixed points. This map is the inclusion of submanifolds
\[\big(M^H\backslash(\bigcup_{j\notin U}P_j)^H\big)\backslash(\bigcap_{u\in U}P_u)^H\longrightarrow M^H\backslash(\bigcup_{j\notin U}P_j)^H\]
which is as connected as the codimension of $(\bigcap_{u\in U}P_u)^H$ in $M^H$ minus one. By our transversality assumption $(\bigcap_{u\in U}P_u)^H=\bigcap_{u\in U}P_u\cap M^H$ is a submanifold of $M^H$ of codimension $\sum_{u\in U}\big(m_H-d_u(H)\big)$. The functions
\[\nu^U(H)=\sum_{u\in U}\big(m_H-d_u(H)\big)-1\]
satisfy $\nu^U\leq \nu^V$ when $U\subset V$ are both $H$-invariant, and Theorem \ref{BM} applies. The first term of the minimum of \ref{BM} is
\[\min_{\{T_\alpha\}_{\alpha\in A}\in\Pa_H(J)}\Big\{
\sum_{\alpha\in A}\sum_{u\in T_\alpha}\big(m_H-d_u(H)\big)-|A|\Big\}-|J/H|+1\]
The double sum is independent of the partition, and this minimum is realized for the finest partition of $J$ by $H$-invariant sets. This is the partition of $J$ in $H$-orbits, and the term above is
\[\sum_{j\in J}\big(m_H-d_j(H)\big)-2|J/H|+1\]
The second term of the minimum of the Blakers-Massey formula is
\[\min_{\emptyset\neq U\subset J}\min_{L\in \Eff_H(U)}\big\{\conn M^L\backslash (\bigcup_{u\in U}P_u)^L-|U/L|+1\big\}
\]
The connectivity of the complement of $(\bigcup_{u\in U}P_u)^L=\bigcup_{u\in U}(P_u\cap M^L)$ in $M^L$ is 
\[\conn M^L\backslash (\bigcup_{u\in U}P_u)^L=\min\Big\{\conn M^L,\min_{u\in U}\big(m_L-d_u(L)\big)-1\Big\}\]
The double minimum above is  thus realized when $U=J$, with value
\[\min_{L\in \Eff_H(J)}\big\{\min\Big\{\conn M^L,m_L-1-\max_{j\in J}d_j(L)\Big\}-|J/L|+1\big\}
\]
Moreover $\Eff_H(J)$ is the collection of proper subgroups $L$ of $H$ with $J/L\neq J/H$.
\end{proof}

This Theorem has interesting consequences for equivariant cubes of configuration spaces. What follows is an equivariant version of the multiple disjunction Theorem \cite{GK} in the easy situation where the submanifolds $P_j$ are points. The techniques of \cite{GK} can supposedly be extended to equivariant collections of higher dimensional submanifolds.
Given a finite $G$-set $J$ and a proper $G$-manifolds $M$, we let $\conf(J,M)$ be the space of ordered configurations of $|J|$-points in $M$, with $G$ acting by conjugation. Explicitly, a configuration is an injective map $\underline{x}\colon J\to M$, which is sent by the homeomorphism associated to $g\in G$ to the composite $g\underline{x}g^{-1}$.
The cube $\conf\big(J\backslash(-),M\big)\colon \mathcal{P}(J)\to Top$ that sends a subset $U$ of $J$ to the configurations of $J\backslash U$-points has a $G$-structure defined by the maps $g\colon \conf\big(J\backslash U,M\big)\to \conf\big(J\backslash gU,M\big)$ which send $\underline{x}\colon U\to M$ to
\[g\underline{x}\colon J\backslash g(U)\stackrel{g^{-1}}{\longrightarrow}J\backslash U\stackrel{\underline{x}}{\longrightarrow}M\stackrel{g}{\longrightarrow}M\]
\begin{cor}\label{conf}
Let $J$ be a finite $G$-set and let $J_+$ be the $G$-set $J$ with an added fixed basepoint. The $J_+$-cube $\conf(J_+\backslash (-),M)\colon \mathcal{P}(J_+)\to Top$ is $\nu$-cartesian, where $\nu$ is the function
\[\nu(H)=\min\Big\{|J|m_H-2|J/H|+1\ ,
\min_{\substack{
L\leq H\\
J/L\neq J/H
}}\big\{\min\{\conn M^L+1,m_L\}-|J/L|\big\}
 \Big\}\]
and $m_H$ is the dimension of the fixed points manifold $M^H$.
\end{cor}

Because of the second term of the minimum the formula is meaningful only when the connectivity of $M$ is greater than the number of points in the configuration. Therefore one cannot expect to make $\nu$ diverge by increasing the number of orbits of $J$. If the action on $J$ is trivial, the second term is infinite and the range of \ref{conf} is the classical cartesianity range from embedding calculus (see \cite{GK}).
Before proving Corollary \ref{conf} we see an example where the range is sharp and it is determined by the second term of the minimum. 

\begin{ex}\label{exsharp}
Let us consider the group $G=\mathbb{Z}/2$ and the finite $G$-set $J=\mathbb{Z}/2$ with action by left multiplication. The space $\conf(\mathbb{Z}/2_+,M)$ of configurations in a manifold with involution $M$ is the space of triples of pairwise distinct elements $(x_+,x_0,x_1)$ in $M$. Such a triple is sent by the non-trivial element of $\mathbb{Z}/2$ to $(\tau x_{+},\tau x_{1},\tau x_{0})$, where $\tau$ denotes the involution of $M$. The $\mathbb{Z}/2$-fixed points of this space is described by a natural homeomorphism
\[\conf(\mathbb{Z}/2_+,M)^{\mathbb{Z}/2}\cong M^{\mathbb{Z}/2}\times (M\backslash M^{\mathbb{Z}/2})\]
The equivariant cube $\conf(\mathbb{Z}/2_+\backslash(-),M)$ is the $\mathbb{Z}/2_+$-cube
\[\xymatrix@=8pt{\conf(\mathbb{Z}/2_+,M)\ar[rr]\ar[dd]\ar[dr]
&&\conf(0_+,M)\ar@{<-->}[dl]_-{\tau}\ar[dr]\ar@{|}[d]\\
&\conf(1_+,M)\ar[dd]\ar[rr]&\ar[d]&M\ar[dd]\\
\conf(\mathbb{Z}/2,M)\ar@{-}[r]\ar[dr]&\ar[r]&M\ar@{<-->}[dl]_-{\tau}\ar[dr]\\
& M\ar[rr]&&\ast
}\]
where $0$ and $1$ are the elements of $J=\mathbb{Z}/2$. The dashed maps denote the $G$-structure on the non-fixed objects of $\mathcal{P}(\mathbb{Z}/2_+)$. By inspection, a point in the homotopy limit of this cube with the initial vertex removed is the data of three paths $\gamma_+,\gamma_0,\gamma_1\colon I\to M$, subject to the conditions $\gamma_+(0)\neq \gamma_0(0)$, $\gamma_0(1)\neq \gamma_1(1)$ and $\gamma_1(0)\neq \gamma_+(1)$. The non-trivial element of $\mathbb{Z}/2$ sends $(\gamma_+,\gamma_0,\gamma_1)$ to $(\tau\overline{\gamma}_+,\tau\gamma_1,\tau\gamma_0)$, where the bar denotes the backwards path.
Such a triple of paths is fixed by the action precisely when $\gamma_+$ is a loop in the fixed points manifold $M^{\mathbb{Z}/2}$, and $\gamma_1=\overline{\gamma}_0$.
The fixed points of this homotopy limit is therefore homeomorphic to the space
\[\big(\holim_{\mathcal{P}_0(\mathbb{Z}/2_+)}\conf(\mathbb{Z}/2_+\backslash(-),M)
\big)^{\mathbb{Z}/2}\cong \Big\{\big(\sigma\in\Lambda M^{\mathbb{Z}/2},\gamma\colon I\to M\big) \ |\ \gamma(0)\notin M^{\mathbb{Z}/2}, \gamma(1)\neq\sigma(\ast)\Big\}\]
Here $\Lambda M$ is the space of free loops in $M$.
The canonical map from $\conf(\mathbb{Z}/2_+,M)^{\mathbb{Z}/2}$ to the fixed points of the homotopy limit sends a pair $(x,y)\in M^{\mathbb{Z}/2}\times (M\backslash M^{\mathbb{Z}/2})$ to the corresponding constant paths. By Corollary \ref{conf} it is at least $\nu(\mathbb{Z}/2)$-connected.

Let us choose $M$ to be the torus $S^1\times S^1$ with the involution that swaps the two product factors. The fixed points of the action is a copy of $S^1$ embedded diagonally, and the range of Corollary \ref{conf} is
\[\nu(\mathbb{Z}/2)=\min\left\{1\cdot 2-2\cdot 1+1
\ , \ \min\{1,2\}-2\right\}=\min\big\{
1\ , -1\big\}=-1\]
Let us show that the cube is not $0$-cartesian. The fixed points of the configuration space is
\[\conf(\mathbb{Z}/2_+,S^1\times S^1)^{\mathbb{Z}/2}\cong S^1\times \big((S^1\times S^1)\backslash \Delta\big)\simeq S^1\times (S^1\vee S^1)\]
which is $0$-connected. We need to show that the fixed points space of the homotopy limit has non-trivial $\pi_0$. There is projection map \[\big(\holim_{\mathcal{P}_0(\mathbb{Z}/2_+)}\conf(\mathbb{Z}/2_+\backslash(-),M)
\big)^{\mathbb{Z}/2}\longrightarrow \Lambda M\]
which is split by any choice of path in $M\backslash M^{\mathbb{Z}/2}$. Hence $\pi_0$ of the fixed points of the homotopy limit contains a copy of $\pi_0\Lambda M$, which is non-trivial as $M=S^1\times S^1$ has non-trivial $\pi_1$. 
\end{ex}

\begin{proof}[Proof of \ref{conf}]
Let $\underline{x}\colon J_+\to M$ be a configuration. Its restrictions $\underline{x}_{J_+\backslash U}\in \conf(J_+\backslash U,M)$ for the subsets $U$ of $J_+$ define a basepoint of the diagram $\conf(J_+\backslash (-),M)$. Let $\mathcal{F}_{\underline{x}}(M)\colon \mathcal{P}(J)\to \Top$ be the $J$-cube defined by the homotopy fibers of the maps that forget the basepoint
\[\mathcal{F}_{\underline{x}}(M)_U=\hof\big(\conf(J_+\backslash U,M)\longrightarrow \conf(J\backslash U,M)\big)\]
over the points $\underline{x}_{J\backslash U}$, with the induced canonical maps. The cube $\mathcal{F}_{\underline{x}}(M)$ has a $G_{\underline{x}}$-structure, where $G_{\underline{x}}$ is the stabilizer group of the configuration $\underline{x}$, defined by the canonical maps
\[\xymatrix{
\mathcal{F}_{\underline{x}}(M)_U\ar@{-->}[d]\ar[r]&\conf(J_+\backslash U,M)
\ar[r]\ar[d]^{g}&
\conf(J\backslash U,M)\ni\underline{x}_{J\backslash U} \ar[d]^{g}\ar@{|->}@<8ex>[d]\\
\mathcal{F}_{\underline{x}}(M)_{gU}\ar[r]&\conf(J_+\backslash gU,M)
\ar[r]&
\conf(J\backslash gU,M)\ni\underline{x}_{J\backslash gU}
}\]
The homotopy fiber of the map $\conf(J_+,M)\to \holim_{\mathcal{P}_0(J_+)}\conf(J_+\backslash (-),M)$ over a configuration $\underline{x}$ is $G_{\underline{x}}$-equivalent to the homotopy fiber of the map $\mathcal{F}_{\underline{x}}(M)_\emptyset\to \holim_{\mathcal{P}_0(J)}\mathcal{F}_{\underline{x}}(M)$ over the canonical basepoint defined by $\underline{x}$. Hence $\nu(H)$ is the minimum of the cartesianities of $\mathcal{F}_{\underline{x}}(M)$, for $\underline{x}$ running through the configurations in $\conf(J_+\backslash U,M)$ which are fixed by $H$.
The restriction map $\conf(J_+\backslash U,M)\to \conf(J\backslash U,M)$ is a fibration of $G_U$-spaces, whose fiber over $\underline{x}_{J\backslash U}\colon J\backslash U\to M$ is the manifold $M$ with the image of $\underline{x}_{J\backslash U}$ removed. Hence the cube $\mathcal{F}_{\underline{x}}(M)$ is equivalent as a $G_{\underline{x}}$-diagram to the $J$-cube
\[F_{\underline{x}}(M)_U=M\backslash \bigcup_{j\in J\backslash U}x_j\]
The collection of $0$-dimensional submanifolds $\{x_j\}_{j\in J}$ satisfies the transversality conditions of Theorem \ref{intsubman}, and the cube $F_{\underline{x}}(M)=M\backslash \underline{x}_{\bullet}$ is then
$\nu(H)$-cartesian.

\end{proof}

\section{The equivariant Quillen Theorem $B$}

Quillen's Theorem $B$ of \cite{Quillen} shows that under certain conditions the homotopy fiber of the geometric realization of a functor $F\colon C\to D$ over an object $d$ in $D$ is weakly equivalent to the geometric realization of the over category $F/d$. It is not immediately clear how this result can be generalized to an equivariant context. The analogous statement for an equivariant functor between categories with $G$-actions can easily be reduced to Theorem $B$ by taking fixed points.
The generalization presented in this paper gives a categorical model for the total homotopy fiber of the geometric realization of a $J$-cube of categories. When $J$ is the set with one element, one recovers Quillen's Theorem $B$. 

The key to our equivariant Theorem $B$ is a categorical model for the $G$-homotopy limit of the nerve of a suitable $G$-diagram of categories. This result is not limited to diagrams of cubical shape, but it applies more generally to categories which satisfy a suitable finiteness condition (see \S\ref{secBI}). Even when $G$ is the trivial group our categorical model for the homotopy limit generalizes Theorem $B$ considerably, from homotopy fibers to all finite homotopy limits. Since this result might be of interest to non-equivariant homotopy theorists we prove it as a separate statement.

The section is organized as follows. In \S\ref{secmodel} we define a categorical model for the homotopy limit of the nerve of a diagram $X\colon I\to Cat$, as well as a quasi-fibrancy condition analogous to the hypothesis of Quillen's Theorem $B$. In \S\ref{secBI} we prove that the classifying space of the category constructed in \S\ref{secmodel} models indeed the homotopy limit of $BX$, if $X$ is quasi-fibrant. We call this result ``Theorem $B^I$''. In Corollary \ref{higherQ} we use this result to give a categorical model for the total homotopy fiber of a cube of categories. We think of this Corollary as a Theorem $B$ for cubes. Section \ref{secBGI} contains the generalization of Theorem $B^I$ to $G$-diagrams, and the equivariant Quillen Theorem $B$.

\subsection{Reedy quasi-fibrant diagrams of categories}\label{secmodel}

Let $I$ be a small category and let  $K,X\colon I\to Cat$ be diagrams of small categories. The natural transformations from $K$ to $X$ form a category $\Hom(K,X)$. An object of this category is a natural transformation $\Phi\colon K\to X$, and a morphism $\Lambda\colon \Phi\to \Phi'$ is a pair of natural transformations $\Lambda_1\colon K\to K$ and $\Lambda_2\colon X\to X$ such that the square 
\[\xymatrix@=15pt{K\ar[r]^-{\Lambda_1}\ar[d]_-{\Phi}&K\ar[d]^-{\Phi'}
\\ X\ar[r]_-{\Lambda_2}&X
}\]
commutes.

\begin{rem} The category $\Hom(K,X)$ was introduced in \cite{Lydakis} where the author shows, among other homotopical properties of this construction, that its nerve is isomorphic to the simplicial mapping space of natural transformations $\Hom(NK,NX)$. In particular when $K$ is the functor $K=I/_{(-)}\colon I\to Cat$ the nerve of $\Hom\big(I/_{(-)},X\big)$ is isomorphic to the Bousfield-Kan formula (\cite{BK}) for the homotopy limit of $NX$. This construction computes the homotopy limit of $NX$ in the standard model structure of simplicial sets only when $NX$ is pointwise fibrant, that is only in the rare situation when the vertices $X_i$ are groupoids. The goal of Theorem $B^I$ is to find a condition on $X\colon I\to Cat$, weaker than assuming that $X$ is valued in groupoids, for which the nerve of $\Hom\big(I/_{(-)},X\big)$ is equivalent to the homotopy limit of $NX$.
\end{rem}

\begin{ex}\label{modelhpb}
Let $I$ be the poset $\bullet\to\bullet\leftarrow\bullet$. A diagram indexed over this poset is a pullback diagram of categories $C\stackrel{f}{\to} D\stackrel{g}{\leftarrow} E$. There is an isomorphism of categories  
\[\Hom\Big((\bullet\to\bullet\leftarrow\bullet)/_{(-)}\ ,\ C\stackrel{f}{\to} D\stackrel{g}{\leftarrow} E\Big)\cong f\!\downarrow\! g\]
where $f\!\!\downarrow\!\!g$ is the model for the homotopy pullback of Barwick and Kan \cite{ClarkKan}. The objects of $f\!\!\downarrow\!\!g$ are triples $(c,d,\gamma)$ consisting of objects $c\in C$ and $e\in E$, and a zig-zag of morphisms $\gamma=(f(c)\to d \leftarrow g(e))$ in the category $D$. This can also be described as the Grothendieck construction of the functor $f/_{(-)}\times g/_{(-)}\colon D\longrightarrow Cat$.
\end{ex}

Let $i\!<\!I$ be the full subcategory of the under category $i/I$ of non-identity maps with source $i$. Given a diagram $X\colon I\to Cat$ we define $X_{i<}: (i\unders I)\to Cat$ to be the restriction of $X$ along the projection functor $i\unders I\to I$ that sends $i\to j$ to $j$. For every object $i$ of $I$, there is a functor 
\[m_i\colon X_i\longrightarrow \Hom\big((i\unders I)/_{(-)},X_{i<}\big)\]
that sends an object $x$ of $X_i$ to the natural transformation $m_i(x)\colon (i\unders I)/_{(-)}\to X_{i<}$ consisting of the constant functors $m_i(x)_\alpha\colon (i\unders I)/_{\alpha}\to X_{j}$ that send every object to $\alpha_\ast x$. For the purpose of this paper, we say that a functor of small categories is a weak equivalence if its nerve is a weak equivalence of simplicial sets.
\begin{defn}\label{quasiReedy}
A diagram $X\colon I\to Cat$ is Reedy quasi-fibrant if for every object $i$ of $I$ the functor
\[m_i/_{(-)}\colon \Hom\big((i\unders I)/_{(-)},X_{i<}\big)\longrightarrow Cat\]
sends every morphism in the category of natural transformations to a weak equivalence of categories.
\end{defn}

This condition is reminiscent of the Reedy fibrancy conditions of diagrams indexed over a Reedy category. We explain the relationship between the two conditions more closely. The functor $m_i$ factors through the categorical limit
\[m_i\colon X_i\longrightarrow\lim\limits_{i\stackrel{\neq\id}{\to}j}X_j=\Hom(\ast,X_{i<})\longrightarrow \Hom\big((i\unders I)/_{(-)},X_{i<}\big)\]
where the second map is induced by the projection $(i\unders I)/_{(-)}\to\ast$. The first functor is the $i$-matching functor of $X$. The nerves of the $m_i$'s are thickenings of the matching maps of $NX$. The diagram $NX$ would be Reedy fibrant if the matching maps were Kan fibrations. If $X$ is Reedy quasi-fibrant then, by Quillen's Lemma \cite[p.98]{Quillen} and Thomason's Theorem \cite{Thomason}, the replacement of $m_i$ by the Grothendieck construction
\[NX_i\simeq N\big(\Hom\mathlarger{\mathlarger{\wr}} m_i/_{(-)}\big)\longrightarrow N\Hom\big((i\unders I)/_{(-)},X_{i<}\big)\]
is a quasi-fibration. In this sense the condition of Theorem $B^I$ is a Reedy quasi-fibrancy condition, hence the terminology.

\begin{ex}\label{Reedyhpb}
If we return to the diagram of categories $C\stackrel{f}{\to} D\stackrel{g}{\leftarrow} E$ of Example \ref{modelhpb} we see that this is Reedy quasi-fibrant precisely when the functors $f/_{(-)},g/_{(-)}\colon D\to Cat$ send every morphism to a weak equivalence.
\end{ex}

\subsection{Finite homotopy limits of categories and higher Quillen's Theorem $B$}\label{secBI}

Let $I$ be a small category, and suppose that the nerves of the under categories $N(i/I)$ are finite dimensional simplicial sets, for every object $i$ of $I$. We call a category $I$ with this property left-finite. Such a category has a canonical degree function $Ob I\to\mathbb{N}$ that sends $i$ to the dimension of $N(i/I)$. This is the length of the longest sequence of non-identity morphisms starting at $i$. 
The degree function induces a filtration 
\[I_{\leq 0}\subset I_{\leq 1}\subset \dots \subset I\]
where $I_{\leq n}$ is the full subcategory of $I$ of objects of degree less than or equal to $n$. This filtration is finite precisely when $NI$ is itself finite dimensional.

Let  $X\colon I\to Cat$ be a diagram of categories. We choose the space
\[\holim_INX:=\Hom\big(N(I/_{(-)}),FNX\big)\]
as a model for the homotopy limit of $NX$, where $NX\stackrel{\simeq}{\to} FNX$ is a pointwise fibrant replacement of the diagram $NX$. The fibrant replacement induces a comparison map $N\Hom\big(I/_{(-)},X\big)\to \holim_INX$.

\begin{theoremBI}
Let $I$ be a left-finite category, and let $X\colon I\to Cat$ be a Reedy quasi-fibrant diagram of categories (see Definition \ref{quasiReedy}). There is a weak equivalence
\[\holim_{n\in \mathbb{N}^{op}}N\Hom\big((I_{\leq n})/_{(-)},X_{\leq n}\big)\stackrel{\simeq}{\longrightarrow}\holim_I NX\]
where $X_{\leq n}$ is the restriction of $X$ to $I_{\leq n}$.
In particular if the nerve of $I$ is finite dimensional the map
\[N\Hom\big(I/_{(-)},X\big)\stackrel{\simeq}{\longrightarrow}\holim_I NX\]
is a weak equivalence of simplicial sets.
\end{theoremBI}

\begin{rem}\label{reduceB}
Let us see how Theorem $B^I$ relates to Quillen's Theorem $B$. Let us consider a diagram of categories $C\stackrel{f}{\to} D\stackrel{g}{\leftarrow} E$ on the poset $I=(\bullet\to\bullet\leftarrow\bullet)$. Theorem $B^I$ tells us that the nerve of the category
\[\Hom\Big((\bullet\to\bullet\leftarrow\bullet)/_{(-)}\ ,\ C\stackrel{f}{\to} D\stackrel{g}{\leftarrow} E\Big)\cong f\!\downarrow\! g\]
is equivalent to the homotopy pullback of $Nf$ and $Ng$, if both functors $f/_{(-)},g/_{(-)}\colon D\to Cat$ send morphisms to equivalences (see \ref{modelhpb} for the definition of $f\!\downarrow\! g$). Barwick and Kan arrive to the same conclusion by assuming that just one among the functors $f/_{(-)}$ and $g/_{(-)}$ satisfies this property (see \cite{ClarkKan}). In particular when $E=\ast$ is the trivial category, their result is equivalent to the original formulation of Quillen's Theorem $B$, from \cite{Quillen}.
The fact that one is able to weaken the quasi-fibrancy condition in the pullback case to only one of the functors $f/_{(-)}$ and $g/_{(-)}$ is a special feature of squares. It is completely analogous to the fact that the pullback along a fibration is homotopy invariant, even though the pullback diagram itself is injectively fibrant only when both maps are fibrations (see e.g. \cite[13.3.9]{hirsch}).
\end{rem}

The proof of Theorem $B^I$ is by induction on the filtration $I_{\leq 0}\subset I_{\leq 1}\subset \dots \subset I$, by exploiting the fact that the complements $I_n=I_{\leq n}\backslash I_{\leq n-1}$ are discrete categories. The inductive step is based on a Lemma that describes the interaction between natural transformations and Grothendieck constructions, which requires us to set up some notation.
For any set of degree $n$ objects $U\subset I_n$, let $U\!\leq\! I$ be the union of the under categories $u/I$ for $u\in U$. Explicitly, its set of objects is
\[Ob\ (U\!\leq\! I)=\{(u\in U,\alpha \colon u\to i)\}\]
The set of morphisms $(u,\alpha)\to (v,\beta)$ is empty if $u$ and $v$ are different, and it is the set of morphisms $(u,\alpha)\to (u,\beta)$ in $u/I$ otherwise.
Define $U\unders I$ to be the full subcategory of $U\!\leq\! I$ whose objects are non-identity maps.
Given a diagram of categories $X\colon I\to Cat$, we denote the corresponding restrictions by
\[X_{U\leq}\colon U\!\leq\! I\to I\stackrel{X}{\to} Cat \ \ \ \ \ \ \ \ \ \ \ X_{U<}\colon U\unders I\to I\stackrel{X}{\to} Cat\] where $U\!\leq\! I\to I$ and  $U\unders I\to I$ project onto the target. We recall from \cite{Thomason} that the Grothendieck construction of a functor $F\colon C\to Cat$ is the category $C\wr F$ with objects pairs $(c\in C,x\in \Ob F(c))$, and where a morphism $(c,x)\to (d,y)$ is a morphism $\alpha\colon c\to d$ in $C$ together with a morphism $\delta\colon \alpha_\ast x\to y$ in the category $F(d)$.

\begin{lemma}\label{indgrot}
Let $X\colon I\to Cat$ be a diagram of categories, and suppose that $I$ is left-finite. For every subset $U\subset I_n$, there is a natural isomorphism of categories
\[\Hom\big((U\!\leq\! I)/_{(-)},X_{U\leq}\big)\cong \Big(\Hom\big((U\unders I)/_{(-)},X_{U<}\big)\mathlarger{\mathlarger{\wr}}F_U\Big)\]
where $F_U\colon \Hom\big((U\unders I)/_{(-)},X_{U<}\big)\to Cat$ is the functor that sends a natural transformation $\Phi$ to the category
\[F_U(\Phi)=\prod_{u\in U}(m_u)/_{(\Phi|_{u<I})}\]
\end{lemma}

\begin{proof}
An object in the Grothendieck construction is a collection of functors $\{\Phi_{\alpha}\colon (U\unders I)/_{\alpha}\to X_{i}\}_{\alpha}$ natural in the maps $\alpha\colon u\to i$ ranging over the objects of $U\unders i$, together with  objects $x_u\in X_u$ for every $u\in U$, and compatible natural transformations for every $\alpha\colon u\to i$
\[\gamma_{\alpha}\colon \alpha_\ast x_u\longrightarrow \Phi_\alpha\]
 Here $\alpha_{\ast}x_u\colon (U\unders I)/_{\alpha}\to X_{i}$ is the constant functor with value $\alpha_{\ast}x_u$. Given such an object $(\Phi,\underline{x},\underline{\gamma})$, define a natural transformation $\Psi\colon (U\!\leq\! I)/_{(-)}\to X_{U\leq}$ as follows.

An object of $(U\!\leq\! I)/_{\alpha}$ is a factorization $\vcenter{\hbox{\xymatrix@=5pt{u\ar[dr]\ar[rr]^{\alpha}&&i\\
&k\ar[ur]}}}$, and an object of $(U\unders I)/_{\alpha}$ is a similar factorization where the map $u\to k$ is not an identity. The functor $\Psi_\alpha\colon (U\!\leq\! I)/_{\alpha}\to X_{i}$ is defined on objects by
\[\Psi_\alpha(\vcenter{\hbox{\xymatrix@=5pt{u\ar[dr]
\ar[rr]^{\alpha}&&i\\
&k\ar[ur]}}})=\left\{\begin{array}{ll}\alpha_{\ast}x_u&\ , \mbox{if } (u\to k)=\id_u\\
\Phi_\alpha(\vcenter{\hbox{\xymatrix@=5pt{u\ar[dr]
\ar[rr]^{\alpha}&&i\\
&k\ar[ur]}}})&\ , \mbox{if } (u\to k)\neq \id_u
\end{array}\right.\]
The point here is that $\Phi_\alpha(\vcenter{\hbox{\xymatrix@=5pt{u\ar[dr]
\ar[rr]^{\alpha}&&i\\
&k\ar[ur]}}})$ is defined precisely when $u\to k$ is not an identity.
\vspace{-.4cm}
A morphism $\vcenter{\hbox{\xymatrix@=5pt{u\ar[dr]
\ar[rr]^{\alpha}&&i\\
&k\ar[ur]}}}\to \vcenter{\hbox{\xymatrix@=5pt{u\ar[dr]
\ar[rr]^{\alpha}&&i\\
&l\ar[ur]}}}$ in $(U\!\leq\! I)/_\alpha$ is a map $k\to l$ such that the two relevant triangles commute. Such a morphism is sent to
\[\Psi_\alpha\Big(\!\!\vcenter{\hbox{
\xymatrix@C=5pt@R=8pt{u\ar[ddr]\ar[dr]
\ar[rr]^{\alpha}&&i\\
&k\ar[d]\ar[ur]\\
&l\ar[uur]}}}\Big)=\left\{\begin{array}{ll}
\id_{\alpha_{\ast}x_u}&, \mbox{if } (u\to l)=\id_u
\\
\gamma_{\alpha}\colon \alpha_\ast x_u\rightarrow \Phi_\alpha(\vcenter{\hbox{\xymatrix@=5pt{u\ar[dr]
\ar[rr]^{\alpha}&&i\\
&l\ar[ur]}}})&, \mbox{if } (u\to l)\neq \id_u  ,\ (u\to k)= \id_u
\\
\Phi_\alpha\big(\vcenter{\hbox{\xymatrix@=8pt{k\ar[d]\\ l}}}\big)\colon \Phi_\alpha\big(\vcenter{\hbox{\xymatrix@=5pt{u\ar[dr]
\ar[rr]^{\alpha}&&i\\
&k\ar[ur]}}}\big)\to\Phi_\alpha\big(\vcenter{\hbox{\xymatrix@=6pt{u\ar[dr]
\ar[rr]^{\alpha}&&i\\
&l\ar[ur]}}}\big)
&, \mbox{if } (u\to l)\neq \id_u ,\ (u\to k)\neq \id_u
\end{array}\right.\]
Notice that if $u\to l$ is the identity map on $u$, by degree reasons both $u\to k$ and $k\to l$ must be identities.
This procedure defines a functor 
\[\Big(\Hom\big((U\unders I)/_{(-)},X_{U<}\big)\mathlarger{\mathlarger{\wr}}F_U\Big)\longrightarrow \Hom\big((U\!\leq\! I)/_{(-)},X_{U\leq}\big)\]
on objects. Extend this on morphisms as follows. Unraveling the definitions of the Grothendieck construction and of the natural transformations category, we see that a morphism $(\Phi,\underline{x},\underline{\gamma})\to (\Phi',\underline{x}',\underline{\gamma}')$ in the left-hand category is a collection of compatible natural transformations $\lambda_{\alpha}\colon \Phi_\alpha\to \Phi_{\alpha}'$, for every non-identity map $\alpha\colon u\to i$ with $u\in U$, together with morphisms $f_u\colon x_u\to x_u'$ in $X_u$ for every $u\in U$, which make the squares
\[\xymatrix{\alpha_\ast x_u\ar[r]^{\alpha_\ast f_u}\ar[d]_{\gamma_{\alpha}}&\alpha_\ast x'_u\ar[d]^{\gamma'_{\alpha}}\\
\Phi_\alpha\ar[r]_{\lambda_{\alpha}}&\Phi'_{\alpha}
}\]
commutative. Such a pair $(\lambda,\underline{f})$ induces a morphism $\Psi\to \Psi'$ between the associated natural transformations in $\Hom\big((U\!\leq\! I)/_{(-)},X_{U\leq}\big)$, defined at a non-identity morphism $\alpha\colon u\to i$ by
\[\Psi_\alpha=\Phi_\alpha\stackrel{\lambda_\alpha}{\longrightarrow}\Phi'_\alpha=\Psi'_\alpha\]
and at an identity map $\id_u$ by $f_u\colon \Psi_{\id_u}=\alpha_\ast x_u\to\alpha_\ast x'_u=\Psi'_{\id_u}$.
The resulting functor is an isomorphism of categories. Its inverse sends a natural transformation $\{\Psi_\alpha\colon (U\!\leq\! I)/_{\alpha}\to X_i\}_{\alpha\colon u\to i}$ to the triple $(\Phi,\underline{x},\underline{\gamma})$ consisting of the restrictions $\Phi_\alpha\colon (U\unders I)/_{\alpha}\to (U\!\leq\! I)/_\alpha\stackrel{\Psi_\alpha}{\to}X_i$ for each $(\alpha\colon u\to i)\in U\unders I$, the objects $x_u=(\Psi_u\colon\ast=(U\!\leq\! I)/_{\id_u}\to X_u)$, and the natural transformations $\gamma_{\alpha}$ defined at an object $\vcenter{\hbox{\xymatrix@=5pt{u\ar[dr]\ar[rr]^{\alpha}&&i\\
&k\ar[ur]}}}$ of $(U\!\leq\! I)/_\alpha$ by the morphism in $X_i$
 \[\alpha_\ast x_u=\alpha_\ast\Psi_{\id_u}(\vcenter{\hbox{\xymatrix@=4pt{u\ar@{=}[dr]\ar@{=}[rr]&&u\\
&u\ar@{=}[ur]}}})
=\Psi_\alpha(\vcenter{\hbox{\xymatrix@=5pt{u\ar@{=}[dr]\ar[rr]^{\alpha}&&i\\
&u\ar[ur]_{\alpha}}}})
\stackrel{}{\longrightarrow}
\Psi_\alpha(\vcenter{\hbox{\xymatrix@=5pt{u\ar[dr]\ar[rr]^{\alpha}&&i\\
&k\ar[ur]}}})=\Phi_\alpha(\vcenter{\hbox{\xymatrix@=5pt{u\ar[dr]\ar[rr]^{\alpha}&&i\\
&k\ar[ur]}}})\]
Here the second equality holds by naturality of $\Psi$, and the arrow is $\Psi_{\alpha}$ applied to the morphism $\vcenter{\hbox{\xymatrix@C=9pt@R=7pt{u\ar[ddr]\ar@{=}[dr]
\ar[rr]_{\alpha}&&i\\
&u\ar[d]\ar[ur]\\
&k\ar[uur]}}}$
of $(U\!\leq\! I)/_{\alpha}$ induced by the factorization $\vcenter{\hbox{\xymatrix@=5pt{u\ar[dr]\ar[rr]^{\alpha}&&i\\
&k\ar[ur]}}}$. The inverse can be extended similarly to morphisms.
\end{proof}

\begin{proof}[Proof of Theorem $B^I$]
Let $NX\to FNX$ be a pointwise fibrant replacement of the diagram $NX$.
We prove just below, by induction on $n$, that for every subset $U\subset I_n$ the map
\begin{equation}\label{induction} N\Hom\big((U\!\leq\! I)/_{(-)},X_{U\leq}\big)\cong \Hom\big(N(U\!\leq\! I)/_{(-)},NX_{U\leq}\big)\stackrel{\simeq}{\longrightarrow}\holim_{U\leq I}NX_{U\leq}
\end{equation}
is a weak equivalence. In particular by choosing $U=I_n$ the category $I_n/I$ is $I_{\leq n}$, and this is an equivalence 
\[N\Hom\big((I_{\leq n})/_{(-)},X_{\leq n}\big)\stackrel{\simeq}{\longrightarrow} \displaystyle\holim_{I_{\leq n}} NX_{\leq n}\]
If $NI$ is finite dimensional, then $I=I_{\leq d}$ for some integer $d$, and the map $N\Hom(I/_{(-)},X)\to \displaystyle\holim_{I} NX$ is an equivalence. When $I$ is infinite, taking the homotopy limit over the maps induced by the filtration gives an equivalence
\[\holim_{n\in \mathbb{N}^{op}}N\Hom\big((I_{\leq n})/_{(-)},X_{\leq n}\big)\stackrel{\simeq}{\longrightarrow}\holim_{n\in \mathbb{N}^{op}}\holim_{I_{\leq n}} NX_{\leq n}\]
The structure maps $\holim_{I_{\leq n}} NX_{\leq n}\to \holim_{I_{\leq n-1}} NX_{\leq n-1}$ in the right-hand tower are Kan fibrations. Indeed, they are induced by mapping the cofibrations of diagrams of simplicial sets $\iota_{n}/_{(-)}\to I_{\leq n}$, where $\iota_n\colon I_{\leq n-1}\to I_{\leq n}$ is the inclusion, into the fibrant diagram $FNX_{\leq n}$. Hence the right-hand homotopy limit is equivalent to the categorical limit. Now each $\Hom\big(N(I_{\leq n})/_{(-)},FNX_{\leq n}\big)$ is isomorphic to $\Hom\big(Nj_n/_{(-)},FNX\big)$, where $j_n\colon I_{\leq n}\to I$ is the inclusion. The right-hand limit is then
\[\lim_{n\in \mathbb{N}^{op}}\Hom\big(N(I_{\leq n})/_{(-)},FNX_{\leq n}\big)\cong \Hom\big(\colim_{n}Nj_n/_{(-)},FNX\big)\cong\holim_INX\]
The last isomorphism holds as the category $j_{n}/_{i}$ includes in $j_{n+1}/_{i}$ for every object $i$ of $I$, with union $\bigcup_{n\in\mathbb{N}}j_{n}/_{i}=I/_i$. This will finish the proof of Theorem $B^I$.

We are left with proving the inductive statement (\ref{induction}) above.
The base induction step $n=0$, relies on the fact that for a subset $U\subset I_0$, the category $(U\!\leq\! I)$ is discrete, with objects the identity maps $\id_u$, for $u\in U$. Therefore the category
 $(U\!\leq\! I)/_{\id_u}=\{\id_u\}$ is the one point category. It follows that $\Hom\big((U\!\leq\! I)/_{(-)},X_{U\leq}\big)$ is the product category
\[\Hom\big((U\!\leq\! I)/_{(-)},X_{U\leq}\big)=\prod_{u\in U}X_u\]
and the homotopy limit of $NX_{U\leq}$ is the product
\[\holim NX_{U\leq}=\prod_{u\in U}FNX_u\]
Since the product of simplicial sets preserve all equivalences (not only between fibrant objects), the map $N\prod_{u\in U}X_u\to \prod_{u\in U}FNX_u$ is an equivalence.

Now suppose that $N\Hom\big((U\!\leq\! I)/_{(-)},X_{U\leq}\big)\to\holim NX_{U\leq}$ is an equivalence for every subset $U\subset I_n$, and let $V$ be a subset of $I_{n+1}$. Let $\Phi\colon (V\unders I)/_{(-)}\to X_{V<}$ be a natural transformation, and consider the commutative diagram
\[\xymatrix{
NF_{V}(\Phi)\ar[rr]\ar[d]&&\prod\limits_{v\in V}\hof_{\Phi|_{v<I}}\big( NX_v\to \holim NX_{v<}\big)\ar[d]\\
N \big(\Hom\mbox{\scalebox{1.3}{$\wr$}} F_V\big)\ar[r]^-{\cong}\ar[dr]&N\Hom\big((V\!\leq\! I)/_{(-)},X_{V\leq}\big)\ar[d]\ar[r]&\holim NX_{V\leq}\ar[d]\\
&N\Hom\big((V\unders I)/_{(-)},X_{V<}\big)\ar[r]^-\simeq&\holim NX_{V<}
}\]
The bottom map is an equivalence by the inductive hypothesis, since $V\unders I=U\!\leq\! I$ for the subset of objects $U:=(V\unders I)\cap I_{n}$ of $I_n$. The isomorphism is from Lemma \ref{indgrot}. By our assumption on the diagram $X$, the functor $F_V$ sends every morphism to a weak equivalence. Thus by Quillen's Lemma \cite[p.98]{Quillen} and Thomason's Theorem \cite{Thomason}, the left-hand vertical sequence is a fiber sequence. Let us turn to the right-hand vertical sequence. Let $\iota \colon V\unders I\to V\!\leq\! I$ be the inclusion. The map induced by $\iota$ on homotopy limits is the restriction
\[\holim_{V\!\leq\! I}NX_{V\leq}\longrightarrow \Hom\big(N\iota/_{(-)},FNX_{V\leq}\big)\cong \holim_{V\unders I} NX_{V<}\]
along the inclusion $\iota/_{(-)}\to (V\!\leq\! I)/_{(-)}$. Since this is a cofibration of diagrams of simplicial sets and $FNX_{V\leq}$ is fibrant, the restriction map is a Kan fibration.
Its point fiber over $\Phi$ is the product of total homotopy fibers
\[\prod_{v\in V}\hof_{\Phi|_{v<I}}\big( NX_v\to \holim NX_{v<}\big)\]
and therefore the right-hand vertical sequence in the diagram above is also a fiber sequence.
For finishing our inductive proof, is then enough to show that the map on homotopy fibers
 \[NF_V(\Phi)=\prod_{v\in V}N\big( X_v\stackrel{m_v}{\to} \Hom\big((v\unders I)/_{(-)},X_{v<}\big)\big)/_{\Phi|_{v\unders I}}\longrightarrow \prod_{v\in V}\hof_{\Phi|_{v<I}}\big( NX_v\to \holim NX_{v<}\big)\]
is an equivalence. The product components of this map factor as
\[\xymatrix{N\big(X_v\stackrel{m_v}{\to}\Hom\big((v\unders I)/_{(-)},X_{v<}\big)\big)/_{\Phi|_{v<I}}\ar[r]\ar[dr]&\hof_{\Phi|_{v<I}}\big( NX_v\to N\Hom\big((v\unders I)/_{(-)},X_{v<}\big)\big)\ar[d]\\
&\hof_{\Phi|_{v<I}}\big( NX_v\to \holim NX_{v<}\big)}\]
Our assumption on $X$ says that the functor $m_v/_{(-)}$ sends every map to a weak equivalence. Hence by Quillen's Theorem $B$ the horizontal map at the top of the triangle is an equivalence. The vertical map is also an equivalence, as by hypothesis of induction for the set $U:=(v\unders I)\cap I_{n}$ the map $N\Hom\big((v\unders I)/_{(-)},X_{v<}\big)\big)\to \holim NX_{v<}$ is an equivalence.
\end{proof}

The following result can be thought of as a generalization of Quillen's Theorem $B$ to higher dimensional cubes, with a higher Quillen Theorem $A$ as a consequence.

\begin{cor}\label{higherQ}
Let $n\geq 1$ be an integer, let $X\colon\mathcal{P}(n)\to Cat$ be a Reedy quasi-fibrant cube of categories, and let $\Phi\colon \mathcal{P}_0(-)\to X_{\emptyset <}$ be a natural transformation. The total homotopy fiber of $NX$ over $N\Phi$ is equivalent to the nerve of the over category $m_{\emptyset}/_{\Phi}$. In particular if the categories $m_{\emptyset}/_{\Phi}$ are contractible $NX$ is homotopy cartesian.
\end{cor}
\begin{proof}
Let us recall that the total homotopy fiber of $NX$ over $N\Phi$ is the homotopy fiber
\[\hofib_\Phi\big(NX_{\emptyset}\longrightarrow \holim_{\mathcal{P}_0(n+1)}NX_{\emptyset <}\big)\]
over the element in the homotopy limit defined by $N\Phi$.
Clearly the restriction $X_{\emptyset <}$ of $X$ to $\mathcal{P}_0(n+1)$ is also Reedy quasi-fibrant, and by Theorem $B^I$ the total homotopy fiber is equivalent the homotopy fiber of
\[\hofib_\Phi\big(NX_{\emptyset}\stackrel{Nm_{\emptyset}}{\longrightarrow} N\Hom\big(\mathcal{P}_0(-),X_{\emptyset <}\big)\big)\]
Since $m_{\emptyset}/_{(-)}$ also sends all maps to equivalences, this is equivalent to $Nm_{\emptyset}/_{\Phi}$ by Quillen's Theorem $B$.
\end{proof}

\subsection{The equivariant Quillen Theorem $B$}\label{secBGI}

Let $G$ be a discrete group acting on a category $I$. We generalize the results of the previous section to $G$-homotopy limits of $G$-diagrams of categories $X\colon I\to Cat$. As in \S\ref{secBI} it is going to be convenient to work with diagrams of simplicial sets instead of diagrams of topological spaces. The only difference is that one needs to perform the suitable fibrant replacements. We say that a morphism of $G$-diagrams of simplicial sets $f\colon Z\to Y$ is an equivalence if is an equivalence of $G$-diagrams of topological spaces (as defined in \ref{defeqGdiag}) after taking geometric realizations.

\begin{defn}\label{defHholimsset}
The $G$-homotopy limit of a $G$-diagram of simplicial sets $Y\colon I\to sSet$ is the $G$-simplicial set of natural transformations
\[\holim_IY=\Hom\big(N(I/_{(-)}),FY\big)\]
where $FY$ is a $G$-diagram of simplicial sets with an equivalence $Y\stackrel{\simeq}{\to}FY$, with the property that $(FY)_i$ is a fibrant $G_i$-simplicial set for every object $i$ in $I$.
\end{defn}

It is proved in \cite[2.6]{Gdiags} that such a replacement $Y\stackrel{\simeq}{\to}FY$ always exists, and that the $G$-homotopy limit functor preserves equivalences of $G$-diagrams of simplicial sets. As simplicial mapping spaces with fibrant target commute with geometric realizations, there is a $G$-equivalence $|\holim_IY|\simeq\holim_I|Y|$ relating this construction with the $G$-homotopy limit of \S\ref{secGdiags}.

If $K,X\colon I\to Cat$ are two $G$-diagrams of categories, the category of natural transformations $\Hom(K,X)$ has an induced $G$-action by conjugation, and the isomorphism of simplicial sets $N\Hom(K,X)\cong \Hom(NK,NX)$ is $G$-equivariant. Composing this isomorphism with a fibrant replacement of $X$ leads to a $G$-equivariant map
\[N\Hom\big(I/_{(-)},X\big)\longrightarrow \holim_I NX\]
We extend the quasi-fibrancy condition \ref{quasiReedy} to the equivariant setting, and we show that this map is an equivalence.
Suppose that $I$ is left-finite. The fixed point categories $I^H$ are automatically left-finite for every subgroup $H$ of $G$. For every object $i$ of $I^H$, the under category $i/I$ has an action of $H$, that restricts to the subcategory $i\unders I$.
The restriction of a $G$-diagram $X\colon I\to Cat$ to $i\unders I$ has a canonical structure of $H$-diagram, and the functor $m_i\colon X_i\rightarrow \Hom\big((i\unders I)/_{(-)},X_{i<}\big)$ of Definition \ref{quasiReedy} is $H$-equivariant. Let
\[m^{H}_i\colon X^{H}_i\longrightarrow \Hom\big((i\unders I)/_{(-)},X_{i<}\big)^{H}\] be its restriction to the categories of fixed points.

\begin{defn}\label{defGreedy}
A $G$-diagram of categories $X\colon I\to Cat$ is $G$-Reedy quasi-fibrant if for every subgroup $H$ of $G$ and every object $i$ of $I^H$ the functor $m^{H}_i/_{(-)}$
sends every morphism to a weak equivalence of categories.
\end{defn}

\begin{theoremBIG}
Let $I$ be a left finite category with $G$-action, and let $X\colon I\to Cat$ be a $G$-diagram of categories. If $X$ is $G$-Reedy quasi-fibrant, there is a weak $G$-equivalence
\[\holim_{n\in \mathbb{N}^{op}}N\Hom\big((I_{\leq n})/_{(-)},X_{\leq n}\big)\stackrel{\simeq}{\longrightarrow}\holim_I NX\]
In particular if the nerve of $I$ is finite dimensional the map $N\Hom\big(I/_{(-)},X\big)\stackrel{\simeq}{\longrightarrow}\holim_I NX$
is a weak $G$-equivalence of simplicial $G$-sets.
\end{theoremBIG}

Let $J$ be a finite $G$-set. The following result is our equivariant generalization of Quillen Theorems' $A$ and $B$.

\begin{cor}\label{GQuillen}
Let $X\colon \mathcal{P}(J)\to Cat$ be a $G$-Reedy quasi-fibrant $J$-cube of categories. For every natural transformation $\Phi\colon \mathcal{P}_0(-)\to X_{\emptyset<}$ the nerve of the category $m_{\emptyset}/_{\Phi}$ is $G_{\Phi}$-equivalent to the total homotopy fiber of the cube $NX$ over $N\Phi$. In particular if all the categories $m_{\emptyset}/_{\Phi}$ are $G_{\Phi}$-contractible, $BX$ is a homotopy cartesian $J$-cube of spaces.
\end{cor}

\begin{proof}
It is immediate from Theorem $B^{I}_G$, using the argument of \ref{higherQ}.
\end{proof}

The proof of Theorem $B_{G}^I$ is based on the same inductive argument in the proof of Theorem $B^I$. The key ingredient for the induction step is an equivariant analogue of Lemma \ref{indgrot}. If $Y\colon I\to Cat$ is a $G$-diagram of categories, recall from \ref{GThom} that its Grothendieck construction $I\wr Y$ has an induced $G$-action.
Given a subset $U\subset I_n$, the $G$-action on $I$ induces a $G_U$-action on the categories $U\leq I$ and $U<I$, where $G_U$ is the subgroup of $g$ of elements that send $U$ to itself. The functor $F_U\colon \Hom\big((U\unders I)/_{(-)},X_{U<}\big)\to Cat$ from Lemma \ref{indgrot} that sends $\Phi\colon (U\unders I)/_{(-)}\to X_{U<}$ to 
\[F_U(\Phi)=\prod_{u\in U}(m_u)/_{(\Phi|_{u<I})}\]
has a canonical $G_U$-structure. It is defined by conjugating the $G_U$-action on $U$ indexing the product with the functors
\[\xymatrix{(m_u)/_{(\Phi|_{u<I})}\ar[r]\ar@{-->}[d]&X_u\ar[r]\ar[d]^{g}&\Hom\big((u\unders I)/_{(-)},X_{u<}\big)\ar[d]^g\\
(m_u)/_{(g\Phi|_{u<I})}\ar[r]&X_u\ar[r]&\Hom\big((u\unders I)/_{(-)},X_{u<}\big)
}\]
Hence the Grothendieck construction of $F_U$ inherits a $G_U$-action. The following is immediate.
\begin{lemma}
For every subset $U\subset I_n$, the isomorphism of categories
\[\Hom\big((U\!\leq\! I)/_{(-)},X_{U\leq}\big)\cong \Big(\Hom\big((U\unders I)/_{(-)},X_{U<}\big)\mathlarger{\mathlarger{\wr}} F_U\Big)\]
of Lemma \ref{indgrot} is $G_U$-equivariant.
\end{lemma}

\begin{proof}[Proof of Theorem $B_{G}^I$]
For every group element $g$ of $G$, the automorphism $g$ of $I$ induces an isomorphism of categories $g\colon i/I\to gi/I$. It follows that the nerves $N(i/I)$ and $N(gi/I)$ have the same dimension, and that the degree function $\deg\colon Ob I\to\mathbb{N}$ is $G$-invariant. Hence the $G$-action restricts to the filtration
\[I_{\leq 0}\subset I_{\leq 1}\subset\dots\subset I_{\leq n}\subset\dots\subset I\]
and the $G$-structure on $X\colon I\to Cat$ restricts to a $G$-structure on $X_{\leq n}\colon I_{\leq n}\to Cat$.
Let $NX\stackrel{\simeq}{\to} FNX$ be a pointwise fibrant replacement of $NX$, like in Definition \ref{defHholimsset}. We prove by induction on $n$ that for every subset $U\subset I_n$ the map
\[N\Hom\big((U\!\leq\! I)/_{(-)},X_{U\leq}\big)\cong \Hom\big(N(U\!\leq\! I)/_{(-)},NX_{U\leq}\big)\longrightarrow\holim_{U\leq I}(FNX)_{U\leq}\]
is a weak $G_U$-equivalence. Once this is established, the same argument in the proof of Theorem $B^I$ finishes the proof of $B_{G}^I$.

For $n=0$, the category $U\leq I$ is discrete and the map above is the map of indexed products
\[\prod\limits_{u\in U}NX_u\longrightarrow \prod\limits_{u\in U}FNX_u\]
The fixed points of this map by a subgroup $H\leq G_U$ is isomorphic to the map
\[\prod\limits_{[u]\in U/H}NX_{u}^{H_{u}}\longrightarrow \prod\limits_{[u]\in U/H}FNX_{u}^{H_{u}}\]
for a choice of representatives in each $H$-orbit of $U$, where $H_u$ is the stabilizer group of $u$ in $H$. Each map $NX_{u}^{H_{u}}\to FNX_{u}^{H_{u}}$ is an equivalence of simplicial sets by assumption, and the map above is an equivalence.

Now suppose that the claim is true for $n$, and let $V$ be a subset of $I_{n+1}$. The sequence
\[NF_{V}(\Phi)\longrightarrow N\big(\Hom\mathlarger{\mathlarger{\wr}} F_V\big)\cong N\Hom\big((V\!\leq\! I)/_{(-)},X_{V\leq}\big)\longrightarrow
N\Hom\big((V\!<\! I)/_{(-)},X_{V<}\big)\]
induced by the restriction map is a fiber sequence of simplicial $G_V$-sets. This is because its restriction on fixed points of a subgroup $H\leq G_V$ is the sequence
\[NF_{V}(\Phi)^H\longrightarrow N\big(\Hom\mathlarger{\mathlarger{\wr}} F_V\big)^ H\cong N \big(\Hom^H\mathlarger{\mathlarger{\wr}} F^{H}_V\big)\longrightarrow
N\Hom\big((V\!<\! I)/_{(-)},X_{V<}\big)^H\]
where the functor $F^{H}_V\colon N\Hom\big((V\!<\! I)/_{(-)},X_{V<}\big)^H\to Cat$ sends an $H$-equivariant natural transformation $\Phi$ to
\[F^{H}_V(\Phi)=\Big(\prod_{v\in V}(m_v)/_{(\Phi|_{v<I})}\Big)^H\cong \prod\limits_{[v]\in V/H}m^{H_v}_v/_{(\Phi|_{v<I})}\]
By assumption  $m^{H_v}_v/_{(-)}$ sends every morphism to a weak equivalence, and thus so does $F^{H}_V$. It follows by Lemma \cite[p.98]{Quillen} and \cite{Thomason} that $NF^{H}_V$ is indeed the homotopy fiber of the restriction map. The restriction map
\[\holim NX_{V\leq}\longrightarrow \holim NX_{V<}
\]
is a fibration of simplicial $G$-sets by an argument analogous to the one in the proof of Theorem $B^I$. Its fiber is the product of homotopy fibers $\prod\limits_{v\in V}\hof_{\Phi|_{v<I}}\big( NX_v\to \holim NX_{v<}\big)$. Therefore it is sufficient to show that the map on homotopy fibers
\[NF_V(\Phi)\longrightarrow \prod\limits_{v\in V}\hof_{\Phi|_{v<I}}\big( NX_v\to \holim NX_{v<}\big)\]
is a $G_V$-equivalence. By taking fixed points, this is the case if for every $v\in V$ the map
\[Nm_v/_{(\Phi|_{v<I})}\longrightarrow\hof_{\Phi|_{v<I}}\big( NX_v\to \holim NX_{v<}\big)\]
is a $G_v$-equivalence. This map factors as
\[Nm_v/_{(\Phi|_{v<I})}\to\hof_{\Phi|_{v<I}}\Big( NX_v\to N\Hom\big((v\!<\!I)/_{(-)},X_{v<}\big)\Big)\to \hof_{\Phi|_{v<I}}\big( NX_v\to \holim NX_{v<}\big)\]
The first map is a $G_v$-equivalence, since $m^{H}_v/_{(-)}$ sends every morphism to a weak equivalence of categories for every subgroup $H$ of $G_v$. The second map is also a $G_v$ equivalence, as the map $N\Hom\big((v<I)/_{(-)},X_{v<}\big)\to\holim NX_{v<}$ is a $G_v$-equivalence by the inductive hypothesis.
\end{proof}

\appendix
\section{Appendix}
\subsection{Connectivity of homotopy limits}

We prove a result about the connectivity of the space of natural transformations between two diagrams of spaces. This result was used in Proposition \ref{resfixespts} to calculate the connectivity of the restriction map on $G$-homotopy limits. Let $G$ be a discrete group acting on a small category $I$.

\begin{prop}\label{connfixedptsholim}
Let $K\colon I\to \Top_\ast$ be a $G$-diagram of pointed spaces, cofibrant in the model structure of \cite[2.6]{Gdiags}. Suppose that for every object $i$ of $I$ the simplicial set $NI/_i$ is finite dimensional, and that $K_i$ is a $G_i$-$CW$-complex. Then for every $G$-diagram of pointed spaces $X\colon I\to \Top_\ast$, the $G$-fixed points of the space of natural transformations $\Hom_\ast(K,X)^G$ is
\[\min_{i\in \Ob I}\min_{\substack{H\leq G_i\\ K^{H}_i\neq\ast}}\big(\conn X^{H}_i-\dim K_{i}^H\big)\]
connected.
\end{prop}

\begin{cor}\label{connholim}
Let $X\colon I\to \Top_\ast$ be a $G$-diagram of spaces and suppose that $NI/_i$ is finite dimensional for every object $i$ in $I$. Then the fixed points space $(\holim_IX)^G$ is
\[\min_{i\in\Ob I}\min_{H\leq G_i}\big(\conn X^{H}_i-\dim NI/_{i}^H\big)\]
connected. 
\end{cor}

\begin{rem}
For the trivial group $G=1$, this Corollary shows that the homotopy limit of a diagram of pointed spaces $X\colon I\to \Top_\ast$ is
\[\min_{i\in\Ob I}\big(\conn X_i-\dim NI/_{i}\big)\]
connected.
A similar statement can easily be deduced for diagrams of unpointed spaces, provided the homotopy limit is non-empty. This result seems to be well known by the experts, but the author was not able to find a proof in the literature.
\end{rem}

\begin{proof}[Proof of \ref{connholim}]
The $G$-diagram $NI/_{(-)}$ is cofibrant in $\Top_{a}^{I}$ by \cite[2.19]{Gdiags}. Therefore the $G$-diagram of pointed spaces $(NI/_{(-)})_+$ is cofibrant in $(\Top_{\ast})_{a}^{I}$, and Proposition \ref{connfixedptsholim} for the space of natural transformations of pointed functors
\[\Hom_\ast\big((NI/_{(-)})_+,X\big)^G=\Hom\big(NI/_{(-)},X\big)^G=(\holim_IX)^G\]
gives the formula of the statement.
\end{proof}

\begin{proof}[Proof of \ref{connfixedptsholim}]
For any pointed map $g\colon S^{k}\rightarrow \Hom_\ast(K,X)^G$, with $k$ smaller than the range of the statement, we need to build an extension of $g$ to the $(k+1)$-disc. By the adjunction between pointed mapping spaces and smash products, this is the same as solving the extension problem
\[\xymatrix{K\wedge D^{k+1}\ar@{-->}[r]^-{\widetilde{f}}& X\\
K\wedge S^{k}\ar[u]\ar[ur]_-{\widetilde{g}}
}\]
in the category of $G$-diagrams $(Top_{\ast})_{a}^I$. We define the extension $\widetilde{f}$ by induction on a filtration of the objects of $I$ induced by the degree function $\deg\colon ObI\rightarrow \mathbb{N}$ defined as the dimension of the over categories
\[\deg i=\dim NI/_i\]
This degree function is dual to the one used in Theorem $B^I$.
It is easy to verify that for any non-identity map $i\rightarrow j$ the inequality $\deg(i)<\deg(j)$ holds, and that the degree function is constant on $G$-orbits. This is an equivariant version of a directed Reedy category. For every positive integer $d$, define $I_{\leq d}$ to be the full subcategory of $I$ on objects of degree less than or equal to $d$, and $I_d$ the full subcategory of objects of degree $d$.
Notice that the $G$-action restricts to these categories, and that $I_d$ is a discrete category.

For $i$ of degree $-1$, the category $I_{\leq -1}$ is empty and $\widetilde{f}$ is the empty map.
Now suppose that $\widetilde{f}$ is defined as a natural transformation from the category $I_{\leq d-1}$. We start by defining $\widetilde{f}_i$ on representatives of the orbits of the $G$-action on $I_d$. Let $s\colon I_d/G\rightarrow I_d$ be a section for the quotient map, and let $z$ be an orbit in $I_d/G$.
By degree reasons, the only morphisms of $I_{\leq d}$ involving $s(z)$ are maps $j\rightarrow s(z)$ with $j$ in $I_{\leq d-1}$. In order to be compatible with $I_{\leq d-1}$ and to extend $g$, the map $\widetilde{f}_{s(z)}$ needs to satisfy the following extension problem in $Top_\ast$
\[\xymatrix{ K_{s(z)}\wedge S^{k}\ar[r]\ar[dr]_{\widetilde{g}_{s(z)}}&K_{s(z)}\wedge D^{k+1}\ar@{-->}[d]^-{\widetilde{f}_{s(z)}}&L_{s(z)}(K)\wedge D^{k+1}\ar[l]\ar[d]^-{\widetilde{f}|_{I_{\leq d-1}}}\\
&X_{s(z)}&L_{s(z)}(X)\ar[l]
}\]
Here $L_i(Z)$ is the $i$-latching space of a $G$-diagram $Z\in (Top_{\ast})_a^I$, with vertices $L_i(Z)=\colim_{j\stackrel{\neq\id}{\rightarrow} i}Z_j$.
The right-hand square expresses the compatibility of $\widetilde{f}_{s(z)}$ with the extension already defined on $I_{\leq d-1}$. The stabilizer group $G_{s(z)}$ acts on all the spaces of the diagram, and both horizontal maps are cofibrations of $G_i$-spaces by cofibrancy of $K$.
The extension problem above is equivalent to the extension problem of $G_{s(z)}$-spaces
\[\xymatrix{ L_{s(z)}(K)\wedge D^{k+1}\coprod\limits_{L_{s(z)}(K)\wedge S^k}K_{s(z)}\wedge S^k\ar[d]\ar[r]^-{}&X_{s(z)}\\
K_{s(z)}\wedge D^{k+1}\ar@{-->}[ur]_-{\widetilde{f}_{s(z)}}
}\]
and the vertical map is also a cofibration of pointed $G_i$-spaces. The extension $\widetilde{f}_{s(z)}$ can be defined inductively on the relative cells of the cofibration, provided that for any $G_i/H_+\wedge D^{n+1}$-cell the composition of $\widetilde{g}_{s(z)}$ with the attaching map $G_i/H_+\wedge S^{n}\rightarrow X_i$ is $G_i$-equivariantly null-homotopic. If $K_i\wedge D^{k+1}$ has a $G_i/H_+\wedge D^{n+1}$-cell, its fixed points space $K^{H}_i\wedge D^{k+1}$ has an $(n+1)$-cell, and by dimension reasons we must have
\[n+1\leq \dim K^{H}_i\wedge D^{k+1}=\dim K^{H}_i+k+1\leq \conn X_{i}^H+1\]
The last inequality holds as $k$ is smaller than the range of the statement. Thus $\pi_n X^{H}_i$ is trivial, and any map $(G_i/_H)_+\wedge S^{n}\rightarrow X_i$ is null-homotopic.

Now that $\widetilde{f}_{s(z)}$ is defined on the representatives of the $G$-orbits of $I_d$, we extend it to the rest of $I_d$ by defining
\[\widetilde{f}_i\colon K_i\wedge D^{k+1}\stackrel{g^{-1}}{\longrightarrow}K_{s[i]}\wedge D^{k+1}\stackrel{\widetilde{f}_{s[i]}}{\longrightarrow}X_{s[i]}\stackrel{g}{\longrightarrow}X_{i}\]
for a choice of $g$ in $G$ such that $gs[i]=i$. The map $\widetilde{f}_i$ does not depend on the choice of $g$ because $\widetilde{f}_{s[i]}$ is $G_i$-equivariant. Moreover the compatibility of $\widetilde{f}_{s[i]}$ with the maps $j\rightarrow s(z)$ insures that $\widetilde{f}$ is natural on $I_{\leq d}$. It is easy to verify that $\widetilde{f}$ extends $\widetilde{g}$, and that it is compatible with the $G$-structure.
\end{proof}

\bibliographystyle{amsalpha}
\bibliography{Gdiagsofspaces}

\end{document}